\documentclass{amsart}
\usepackage{amsmath,amssymb,graphicx,amsthm,bm,hyperref,mathrsfs}
\usepackage{graphicx}
\usepackage{hyperref}
\hypersetup{colorlinks=true,citecolor=blue,linkcolor=cyan}
\newtheorem{thm}{Theorem}[section]
\newtheorem{cor}[thm]{Corollary}
\newtheorem{lem}[thm]{Lemma}
\newtheorem{prop}[thm]{Proposition}
\theoremstyle{definition}
\newtheorem{defn}[thm]{Definition}
\theoremstyle{remark}
\newtheorem{rem}[thm]{Remark}
\numberwithin{equation}{section}

\newcommand{\F}{\mathcal{F}}
\newcommand{\Ff}{\mathbb{F}}
\newcommand{\R}{\mathcal{R}}

\newcommand{\Gal}{\mbox{Gal}}
\newcommand{\Z}{\mathbb{Z}}
\newcommand{\Hh}{\mathcal{H}}

\newcommand{\Prob}{\mbox{Prob}}
\newcommand{\Pp}{\mathbb{P}}

\newcommand{\Mod}[1]{\ (\text{mod}\ #1)}

\newcommand{\lcm}{\mbox{lcm}}
\newcommand{\M}{\mathcal{M}}
\newcommand{\Res}{\mbox{Res}}
\begin{document}

\title[Number of Points on the Full Moduli Space of Curves over Finite Fields]{Number of Points on the Full Moduli Space of Curves over Finite Fields}%
\author{Patrick Meisner}%


\begin{abstract}

The distribution of the number of points on abelian covers of $\Pp^1(\Ff_q)$ ranging over an irreducible moduli space has been answered in a recent work by the author \cite{M1},\cite{M2}. The authors of \cite{BDFK+} determined the distribution over the whole moduli space for curves with $\Gal(K(C)/K)$ a prime cyclic. In this paper, we prove a result towards determining the distribution over the whole moduli space of curves with $\Gal(K(C)/K)$ any abelian group. We successfully determine the distribution in the case $\Gal(K(C)/K)$ is a power of a prime cyclic.
\end{abstract}
\maketitle


\section{Introduction}\label{intro}

Let $\Hh$ be a family of smooth, projective curves over $\Ff_q$, the finite field with $q$ elements. We are interested in determining the probability that a curve, chosen randomly from our family, has a given number of points. Classical results due to Katz and Sarnak \cite{KS} tell us what happens if we fix the genus of the curve, $g$ and let $q\to\infty$. Progress has been made in the other situation when $q$ is fixed and $g\to\infty$.

Let $K=\Ff_q(X)$ and $K(C)$ be the field of functions of $C$, a curve over $\Ff_q$. Then $K(C)$ will be a finite field extension of $K$. Moreover, if we fix a copy of $\Pp^1(\Ff_q)$, then every such finite extension corresponds to a smooth, projective curve (Corollary 6.6 and Theorem 6.9 from Chapter I of \cite{hart}). If $K(C)$ is a Galois extension of $K$, denote $\Gal(C):=\Gal(K(C)/K)$ and $g(C)$ to be the genus of $C$. Define
$$\Hh_{G,g} = \{C: \Gal(C)=G, g(C)=g\}.$$

We want to determine the probability, that a random curve in this family has a given number of points. That is, for every $N\in \Z_{\geq 0}$, we want to determine
$$\Prob(C\in\Hh_{G,g} : \#C(\Pp^1(\Ff_q)) = N) := \frac{|\{C\in \Hh_{G,g} :  \#C(\Pp^1(\Ff_q)) = N\}|}{|\Hh_{G,g}|}.$$

Therefore, the first thing we need to do is determine $|\Hh_{G,g}|$. Wright \cite{wright} was the first to answer such a question. He proved that if $G$ is abelian and $q\equiv 1 \Mod{\exp(G)}$, (where $\exp(G)$ is the smallest integer such that $ng=e$ for all $g\in G$) then as $g\to\infty$
\begin{align}\label{wright}
\sum_{j=0}^{N-1}q^{-\frac{j}{N}}|\Hh_{G,g+j}| \sim C(K,G)g^{\phi_G(Q)-1}q^{\frac{g}{N}}
\end{align}
where $C(K,G)$ is a non-zero constant, $N = |G|-\frac{|G|}{Q}$ where $Q$ is the smallest prime divisor of $|G|$ and $\phi_G(Q)$ is the number of elements of $G$ of order $Q$. (Note: Wright's result does not require $q\equiv 1 \Mod{\exp(G)}$, but we will always assume that here and it makes the formula slightly nicer).

\begin{rem}

From now on the function $\phi_G(s)$ will be the number of elements of $G$ of order $s$.

\end{rem}

Bucur, David, Feigon, Kaplan, Lalin, Ozman and Wood \cite{BDFK+} shows that if $Q$ is a prime then as $g\to\infty$
$$|\Hh_{\Z/Q\Z,g}| = \begin{cases} c_{Q-2}q^{\frac{2g+2Q-2}{Q-1}} P\left(\frac{2g+2Q-2}{Q-1}\right) + O\left(q^{(\frac{1}{2}+\epsilon) \frac{2g+2Q-2}{Q-1}  }\right) & g\equiv 0 \mod{Q-1} \\ 0 & \mbox{otherwise}  \end{cases}$$
where $c_{Q-2}$ is a constant that they make explicit and $P$ is a monic polynomial of degree $Q-2$.

Our first result is extending this to any abelian group. But first, we must define a quasi-polynomial.
\begin{defn}

A quasi-polynomial is a function that can be written as
$$p(x) = c_n(x)x^n + c_{n-1}(x)x^{n-1} + \dots + c_0(x)$$
where $c_i(x)$ is a periodic function with integer period. We call the $c_i$ the coefficients of the quasi-polynomial. Moreover, if $c_n(x)$ is not identically the zero function then we say $p$ has degree $n$ and call it the leading coefficient.

\end{defn}

\begin{defn}

Let $\R=[0,\dots,r_1-1]\times\dots\times[0,\dots,r_n-1]\setminus\{(0,\dots,0)\}$ be a set of inter-valued vectors. For any $\vec{\alpha}\in\R$ let
$$e(\vec{\alpha}) = \underset{j=1,\dots,n}{\lcm}\left(\frac{r_j}{(r_j,\alpha_j} \right)$$

\end{defn}

\begin{thm}\label{mainthm1}

Let $G$ be any abelian group and $q\equiv 1 \mod{\exp(G)}$. If there exists some $(d(\vec{\alpha}))_{\vec{\alpha}\in\R}$ such that
$$2g+2|G|-2 = \sum_{\vec{\alpha}\in\R} \left(|G|-\frac{|G|}{e(\vec{\alpha})}\right) d(\vec{\alpha}) + |G|-\frac{|G|}{e(\vec{d})} $$
where $\vec{d}=(d_1,\dots,d_n)$ and
$$d_j = \sum_{\vec{\alpha}\in\R}\alpha_jd(\vec{\alpha})$$
then
$$|\Hh_{G,g}| = \sum_{j=1}^{\eta}P_j(2g)q^{\frac{2g+2|G|-2}{|G|-\frac{|G|}{s_j}}} + O\left(q^{\frac{(1+\epsilon)g}{|G|-\frac{|G|}{s_1}}}\right)$$
where $1=s_0<s_1<\dots<s_{\eta}=\exp(G)$ are the divisors of $\exp(G)$, $P_1$ is a quasi-polynomial of degree $\phi_G(s_1)-1$ and $P_j$ is a quasi-polynomial of degree at most $\phi_G(s_j)-1$. If no such solution exists then $|\Hh_{G,g}|=0$.

\end{thm}

If we restrict to the case $G=(\Z/Q\Z)^n$, then we can say more about the polynomials.

\begin{cor}\label{mainthm1cor}

If $G=(\Z/Q\Z)^n$ for $Q$ a prime, $q\equiv1 \mod{Q}$, and $2g+2Q^n-2\equiv 0\Mod{Q^n-Q^{n-1}}$, then
$$|\Hh_{G,g}| = P\left(\frac{2g+2Q^n-2}{Q^n-Q^{n-1}}\right)q^{\frac{2g+2Q^n-2}{Q^n-Q^{n-1}}} + O\left(q^{\frac{(1+\epsilon)g}{Q^n-Q^{n-1}}}\right)$$
where $P$ is a polynomial of degree $Q^n-2$ with leading coefficient
$$\frac{1}{(Q^n-2)!}\frac{q+Q^n-1}{q}\frac{L_{Q^n-2}}{\zeta_q(2)^{Q^n-1}}$$
where $L_{Q^n-2}$ is a constant defined in \eqref{Lnum} and $\zeta_q(s)$ is the zeta function for $\Ff_q[X]$ (\eqref{zeta}). If $2g+2Q^n-2\not\equiv 0 \Mod{Q^n-Q^{n-1}}$, then $|\Hh_{G,g}|=0$
\end{cor}

If $q\equiv 1 \Mod{\exp(G)}$ and $G$ is abelian, then we can we can write
$$\Hh_{G,g} = \bigcup \Hh^{\vec{d}(\vec{\alpha})}$$
where $\Hh^{\vec{d}(\vec{\alpha})}$ is an irreducible moduli space of $\Hh_{G,g}$ (see Section 2 of \cite{M1} for a full treatment of this).

For specific classes of groups, several authors determined that as $d(\vec{\alpha})\to\infty$ for all $\vec{\alpha}$, then
\begin{align}\label{irresult}
\Prob(C\in\Hh^{\vec{d}(\vec{\alpha})} : \#C(\Pp^1(\Ff_q)) = N) \sim \Prob\left(\sum_{i=1}^{q+1}X_i = N\right)
\end{align}
where the $X_i$ are i.i.d. random variables that can be made completely explicit.

Kurlberg and Rudnick \cite{KR} were the first to do this for hyper-elliptic curves ($G=\Z/2\Z$). Bucur, David, Feigon and Lalin \cite{BDFL1},\cite{BDFL2} then extended this to prime cyclic curves ($G=\Z/Q\Z$, $Q$ a prime). Lorenzo, Meleleo and Milione \cite{LMM} then proved this for $n$-quadratic curves ($G=(\Z/2\Z)^n$). The author \cite{M1},\cite{M2} completes this for any abelian group.

The fact that these results need all the $d(\vec{\alpha})\to\infty$ is why we can not deduce the results for the whole space from the results for the subspaces. However, in the case $G=(\Z/Q\Z)^n$, we can deduce the main term of Corollary \ref{mainthm1cor} from these results. However, the error term we get from doing this is just $(1+o(1))$. Likewise we can do the same for Corollary \ref{mainthm2cor} and Theorem \ref{mainthm3}.

If $G=\Z/r_1\Z\times\dots\times\Z/r_n\Z$, then since we assume $q\equiv1 \mod{\exp(G)}$, by Kummer Theory, we can find $F_j\in\Ff_q[X]$, $r_j^{th}$-power free such that
$$K(C) = K\left(\sqrt[r_1]{F_1(X)}, \dots, \sqrt[r_n]{F_n(X)}\right).$$

Fix an ordering $x_1,\dots,x_{q+1}$ of $\Pp^1(\Ff_q)$ such that $x_{q+1}$ is the point at infinity. This ordering will be fixed for the rest of this paper. Then, the number of points on the curve depend on the values of
$$\chi_{r_j}(F_j(x_i)), j=1,\dots,n, i=1,\dots,q+1$$
where $\chi_{r_j}:\Ff_q\to\mu_{r_j}$ is a multiplicative character of order $r_j$. Moreover, the value of $F_j(x_{q+1})$ depends on the leading coefficient and degree of $F_j(X)$. (Again, see \cite{M2} for more on this.) Therefore, we will define
$$\vec{k}=(k_1,\dots,k_n)\in\Z^n$$
$$E = \begin{pmatrix} \epsilon_{1,1} & \dots & \epsilon_{1,n} \\ \vdots & \ddots & \vdots \\ \epsilon_{\ell,1} & \dots & \epsilon_{\ell,n} \end{pmatrix} $$
such that $0\leq \ell \leq q$, $\epsilon_{i,j}\in \mu_{r_j}$. Define
\begin{align*}
\Hh_{G,g}(\vec{k},E) = \{C\in\Hh_{G,g} : &  \deg(F_j)\equiv k_j \Mod{r_j}, \chi_{r_j}(F_j(x_i))=\epsilon_{i,j},\\
&  i=1,\dots,\ell, j=1,\dots,n\}.
\end{align*}

Then we get

\begin{thm}\label{mainthm2}
Let $G$ be any abelian group and $q\equiv 1 \Mod{\exp(G)}$. If there exists some $(d(\vec{\alpha}))_{\vec{\alpha}\in\R}$ such that
$$2g+2|G|-2 = \sum_{\vec{\alpha}\in\R} \left(|G|-\frac{|G|}{e(\vec{\alpha})}\right) d(\vec{\alpha}) + |G|-\frac{|G|}{e(\vec{k})} $$
where
$$d_j = \sum_{\vec{\alpha}\in\R}\alpha_jd(\vec{\alpha}) \equiv k_j \mod{r_j}$$
then
$$|\Hh_{G,g}(\vec{k},E)| = \sum_{j=1}^{\eta} P_{j;\vec{k},E}(2g)q^{\frac{2g+2|G|-2}{|G|-\frac{|G|}{s_j}}} + O\left(q^{\frac{(1+\epsilon)g}{|G|-\frac{|G|}{s_1}}} \right) $$
where $1=s_0<s_1<\dots<s_{\eta}=r_n$ are the divisors of $r_n$ and $P_{j;\vec{k},E}$ is a quasi-polynomial of degree at most $\phi_G(s_j)-1$. If there no such solution exists then $|\Hh_{G,g}(\vec{k},E)|=0$.
\end{thm}

Again, if $G=(\Z/Q\Z)^n$, then we can say more

\begin{cor}\label{mainthm2cor}

If $G=(\Z/Q\Z)^n$, for $Q$ a prime, $q\equiv 1 \mod{Q}$ and $2g+2Q^n-2\equiv 0 \Mod{Q^n-Q^{n-1}}$ then
$$|\Hh_{(\Z/Q\Z)^n,g}(\vec{k},E)| = P_{\vec{k},E}\left(\frac{2g+2Q^n-2}{Q^n-Q^{n-1}}\right) q^{\frac{2g+2Q^n-2}{Q^n-Q^{n-1}}} + O\left(q^{\frac{(1+\epsilon)g}{Q^n-Q^{n-1}}} \right)$$
where $P_{\vec{k},E}$ is a polynomial of degree $Q^n-2$ with leading coefficient
$$\frac{(q-1)^n}{Q^n(Q^n-2)!}\frac{L_{Q^n-2}}{\zeta_q(2)^{Q^n-1}}\left(\frac{q}{Q^n(q+Q^n-1)}\right)^{\ell}.$$
If $2g+2Q^n-2\not\equiv 0 \Mod{Q^n-Q^{n-1}}$ then $|\Hh_{(\Z/Q\Z)^n,g}(\vec{k},E)|=0$.

\end{cor}

Now, using Corollaries \ref{mainthm1cor} and \ref{mainthm2cor} we can prove a result on the distribution of the number of points for the whole space.

\begin{thm}\label{mainthm3}
Let $G=(\Z/Q\Z)^n$ and fix $q$ such that $q\equiv 1 \Mod{Q}$. If $2g+2Q^n-2\equiv 0\Mod{Q^n-Q^{n-1}}$ then as $g\to\infty$,
$$\frac{|\{C\in\Hh_{(\Z/Q\Z)^n,g} : \#C(\Pp^1(\Ff_q)) = M\}|}{|\Hh_{(\Z/Q\Z)^n,g}|} = \Prob\left(\sum_{i=1}^{q+1} X_i= M\right)\left(1+O\left(\frac{1}{g} \right)\right)$$
where the $X_i$ are $i.i.d.$ random variables taking value $0$, $Q^n$ or $Q^{n-1}$  such that
$$X_i = \begin{cases} Q^{n-1} & \mbox{with probability }  \frac{Q^n-1} {Q^{n-1}(q+Q^n -1)} \\ Q^n & \mbox{with probability }  \frac{q}{Q^n(q+Q^n-1)} \\ 0 & \mbox{with probability } \frac{(Q^n-1)(q+Q^n-Q)}{Q^n(q+Q^n-1)}    \end{cases}.$$
\end{thm}

\begin{rem}

The proof of Theorem \ref{mainthm3} follows directly from Corollaries \ref{mainthm1cor} and \ref{mainthm2cor} and the work done in \cite{M1} and \cite{M2}. Therefore, if we were able to determine the leading coefficients of $P_1$ in Theorem \ref{mainthm1} and $P_{1,\vec{k},E}$ in Theorem \ref{mainthm2}, then an analogous result as Theorem \ref{mainthm3} would follow from the work done in \cite{M1} and \cite{M2}.

\end{rem}

\begin{rem}

The random variables appearing in Theorem \ref{mainthm3} are the same that appear in \eqref{irresult} in the case $G=(\Z/Q\Z)^n$.

\end{rem}

\begin{rem}

Bucur, David, Feigon, Kaplan, Lalin, Ozman and Wood \cite{BDFK+} prove analogous results to Corollary \ref{mainthm2cor}, and Theorem \ref{mainthm3} for $G=\Z/Q\Z$.

\end{rem}


\section{Notation and Setup}

From now on, we will assume that $G=\Z/r_1\Z\times\dots\times\Z/r_n\Z$ such that $r_j|r_{j+1}$ and $q\equiv1 \Mod{r_n}$. Under these assumptions we can apply Kummer theory (Chap.14 Proposition 37 of \cite{DF}) to find $F_j\in\Ff_q[X]$, $r_j^{th}$-power free for $j=1,\dots,n$ such that
$$K(C) = K\left(\sqrt[r_1]{F_1(X)},\dots,\sqrt[r_n]{F_n(X)}\right).$$

Let
$$\Hh^*_{G,g} = \{C\in\Hh_{G,g} : F_j \mbox{ is monic}\}$$
$$\Hh^*_{G,g}(\vec{k},E) = \{C \in \Hh_{G,g}(\vec{k},E): F_j \mbox{ is monic}\}.$$
We call curves in $\Hh^*_{G,g}$ \textit{monic}.

Now for each $F_j$, it's leading coefficient can be chosen from any of the equivalence classes of $\Ff_q^*/(\Ff_q^*)^{r_j}$ to give a different extension (and thus curve). Therefore we see that $|\Hh_{G,g}| = |G||\Hh^*_{G,g}|$ and $|\Hh_{G,g}(\vec{k},E)| = |G||\Hh_{G,g}(\vec{k},E)|$. Therefore, we will work with $\Hh^*_{G,g}$ and $\Hh^*_{G,g}(\vec{k},E)$ from now on.

Now, define
$$\R = [0,\dots,r_1-1]\times\dots\times[0,\dots,r_n-1]\setminus\{(0,\dots,0)\}$$
to be set of integer valued vectors denoted as $\vec{\alpha}=(\alpha_1,\dots,\alpha_n)$ such that not all of them are $0$. Denote
$$\R' = \R\cup\{(0,\dots,0)\}.$$
Then for every $\vec{\alpha}\in\R$ let
$$f_{\vec{\alpha}} = \prod_{\substack{P \\ v_P(F_j)=\alpha_j}} P$$
where the product is over prime polynomials of $\Ff_q[X]$. Then we can write
$$F_j = \prod_{\vec{\alpha}\in\R}f_{\vec{\alpha}}^{\alpha_j}$$
where we use the convention that $f^0$ is identically the constant polynomial $1$. Moreover, all the $f_{\vec{\alpha}}$ are squarefree and pairwise coprime.

In \cite{M2}, the author uses the Riemann-Hurwitz formula (Theorem 7.16 of \cite{rose}) to show that the genus of $C$ satisfies the relation
\begin{align}\label{genform}
2g+2|G|-2 = \sum_{\vec{\alpha}} \left(|G|-\frac{|G|}{e(\vec{\alpha})}\right) \deg(f_{\vec{\alpha}}) + |G|-\frac{|G|}{e(\vec{d})}
\end{align}
where $\vec{d}=(d_1,\dots,d_n)=(\deg(F_1),\dots,\deg(F_n))$ and for any vector $\vec{v}=(v_1,\dots,v_n)$,
\begin{align}\label{ramind}
e(\vec{v}) := \underset{j=1,\dots,n}{\lcm}\left(\frac{r_j}{(r_j,v_j)}\right).
\end{align}

Notice, that the genus only depends on the degree of the $f_{\vec{\alpha}}$ and the congruence class of the $d_j$ modulo $r_j$. Therefore, define
\begin{align}
\F_d = \{f \in \Ff_q[X]: f \mbox{ is monic, squarfree and } \deg(f)=d\}
\end{align}
\begin{align}\label{Fset1}
\F_{\vec{d}(\vec{\alpha})} = \{(f_{\vec{\alpha}})\in\prod_{\vec{\alpha}\in\R} \F_{d(\vec{\alpha})} : (f_{\vec{\alpha}},f_{\vec{\beta}})=1 \mbox{ for all } \vec{\alpha}\not=\vec{\beta}\in\R \}
\end{align}
where $\vec{d}(\vec{\alpha})=(d(\vec{\alpha}))_{\vec{\alpha}\in\R}$ is a vector of non-negative integers indexed by the vectors of $\R$.

Now for any $\vec{k}=(k_1,\dots,k_n)\in\R'$ consider the congruence conditions
\begin{align}\label{equiv1}
\sum_{\vec{\alpha}\in\R} \alpha_jd(\vec{\alpha}) \equiv k_j \mod{r_j}, j=1,\dots,n.
\end{align}
Further, let $E$ be an $\ell\times n$ matrix such that
\begin{align}\label{matrix}
E = \begin{pmatrix} \epsilon_{1,1} & \dots & \epsilon_{1,n} \\ \vdots & \ddots & \vdots \\ \epsilon_{\ell,1} & \dots & \epsilon_{\ell,n} \end{pmatrix}
\end{align}
with $\epsilon_{i,j}\in\mu_{r_j}$. Now, define
\begin{align}\label{Fset2}
\F_{\vec{d}(\vec{\alpha}); \vec{k},E} = \begin{cases} \{(f_{\vec{\alpha}})\in\F_{\vec{d}(\vec{\alpha})} : \chi_{r_j}(F_j(x_i))=\epsilon_{i,j}, i=1,\dots,\ell,j=1,\dots,n\} & \eqref{equiv1} \mbox{ is satisfied} \\ \emptyset & \mbox{otherwise} \end{cases}
\end{align}

Finally, define
\begin{align}\label{Fset3}
\F_{D;\vec{k},E} = \bigcup_{\vec{d}(\vec{\alpha})} \F_{\vec{d}(\vec{\alpha});\vec{k},E}
\end{align}
where the union is over all $\vec{d}(\vec{\alpha})$ that satisfies
\begin{align}\label{genform2}
D = \sum_{\vec{\alpha}\in\R} c(\vec{\alpha})d(\vec{\alpha})
\end{align}
where $c(\vec{\alpha})$ are some fixed constants.

If we set $c(\vec{\alpha}) = |G|-\frac{|G|}{e(\vec{\alpha})}$ and $D = 2g+2|G|-2-c(\vec{k})$, then we see that \eqref{genform2} becomes \eqref{genform}. Outside of Section \ref{curvesec}, we will work with arbitrary $c(\vec{\alpha})$ with the idea that eventually we will set them equal to what we need.

Therefore, it seems as if what we need to do is determine the size of $\F_{D;\vec{k},E}$ for all $D,\vec{k},E$. Unfortunately, this will actually count too many curves! However, there is still a way to determine the size of $\Hh^*_{G,g}$ and $\Hh^*_{G,g}(\vec{k},E)$ if we know $|\F_{D;\vec{k},E}|$.


\section{Too Many Curves}\label{curvesec}

Ideally, we would like to say that every monic curve, $C$, such that $\Gal(C)=G$, $g(C)=g$, comes from an element $\F_{\vec{d}(\vec{\alpha})}$ such that $\vec{d}(\vec{\alpha})$ satisfies \eqref{genform}. Unfortunately, this is not true.

For example, if we consider the set $\F_{(0,d_2,0)}$ such that $2g+6=2d_2$ and $2d_2\equiv 0\Mod{4}$. Then $(0,d_2,0)$ satisfies \eqref{genform} for $G=\Z/4\Z$ and we would hope that this would correspond to a curve with $\Gal(C)=\Z/4\Z$ and $g(C)=g$. However, an element of $\F_{(0,d_2,0)}$ would look like $(1,f_2,1)$ where $f_2$ is a square-free polynomial of degree $d_2$. This would correspond to a curve with affine model $Y^4 = f_2^2$, which clearly has $K(C)=K(\sqrt{f_2})$ and so $\Gal(C)=\Z/2\Z$. Moreover,
$$g(C) = \frac{d_2-2}{2} = \frac{g+3-2}{2} = \frac{g-1}{2}+1.$$

It is easy to see how this argument can be extended to any group $G$ that does not have prime order. Indeed, what we will show in this section is that the elements of $\F_{\vec{d}(\vec{\alpha})}$ correspond to monic curves whose Galois group is a \textit{subgroup} of $G$.

\begin{rem}\label{curverem}

When we talk about all the subgroups of $G$, we mean all the different possible subsets of $G$ that are subgroups of $G$. That is, two subgroups $H,H'\subset G$ are said to be the same subgroup if and only if they are equal as subsets. For example, if $G=\Z/Q\Z\times\Z/Q\Z$, then the subgroups
$$\{(a,0): 0\leq a \leq Q-1\}$$
$$\{(0,a): 0\leq a \leq Q-1\}$$
$$\{(a,a): 0\leq a \leq Q-1\}$$
are all different even though they are all isomorphic to $\Z/Q\Z$.

\end{rem}

For simplicity, in this section, we will fix $c(\vec{\alpha})=|G|-\frac{|G|} {e(\vec{\alpha})}$.

\begin{prop}\label{curveprop1}
Let
$$M(G,g) = \{C, \mbox{ monic } : \Gal(C)=H\subset G, g(C)=\frac{g-1}{|G|/|H|}+1\}.$$
Then there is a natural bijection from elements of
$$\bigcup_{\vec{d}(\vec{\alpha})} \F_{\vec{d}(\vec{\alpha})} $$
to $M(G,g)$ where the union is over all $\vec{d}(\vec{\alpha})$ that satisfies \eqref{genform}.
\end{prop}

\begin{proof}

Let $(f_{\vec{\alpha}}) \in \F_{\vec{d}(\vec{\alpha})}$ for some  fixed $\vec{d}(\vec{\alpha})$ that satisfies \eqref{genform}. Define
$$\R^* = \{\vec{\alpha}\in\R: d(\vec{\alpha})\not=0\}.$$
For every $\vec{\alpha}\in\R$, we can identify it as an element in $G$ in the natural way. Let $H\subset G$ be the subgroup that is generated by $\R^*$ under this identification. (From now on, in this proof, we will identify elements of $H$ and $G$ with elements of $\R$). We will show that $\Gal(C)=H$.

There exists some $s_j|r_j$ (where, potentially, some of the $s_j=1$) such that
$$H\cong \Z/s_1\Z \times \dots \times \Z/s_n\Z.$$

Let $\vec{\alpha}_j \in H$ be a generating set of $H$ such that the order of $\alpha_j$ is $s_j$. Therefore, if $\vec{\alpha}\in\R^*$, we can find $\alpha_j^*$ such that $0\leq \alpha_j^*\leq s_j-1$ and
$$\vec{\alpha} = \sum_{j=1}^n \alpha_j^*\vec{\alpha}_j.$$

If we let $\vec{\alpha}_j=(\alpha_{j,1},\dots,\alpha_{j,n})$ then for all $\vec{\alpha}\in\R^*$,
$$\alpha_k = \sum_{j=1}^n \alpha_j^*\alpha_{j,k}.$$

Now, a basis element of $K(C) = K\left(\sqrt[r_1]{F_1(X)}, \dots, \sqrt[r_n]{F_n(X)} \right)$ will be
$$\prod_{k=1}^n F_k(X)^{\frac{m_k}{r_k}} = \prod_{\vec{\alpha}\in\R^*} f_{\vec{\alpha}}^{\sum_{k=1}^n \frac{\alpha_km_k}{r_k}} = \prod_{\vec{\alpha}\in\R^*} f_{\vec{\alpha}}^{\sum_{k=1}^n \frac{m_k}{r_k} \sum_{j=1}^n \alpha_j^*\alpha_{j,k} } $$
for some values of $m_k,k=1,\dots,n$. (Note, we can restrict the product down to the $\vec{\alpha}\in\R^*$ for if $\vec{\alpha}\not\in\R^*$, then $\deg(f_{\vec{\alpha}})=0$ and hence $f_{\vec{\alpha}}=1$.) Therefore, we can define an action by $h=(h_1,\dots,h_n)\in H$ on the basis elements by
$$h\left(\prod_{k=1}^n F_k(X)^{\frac{m_k}{r_k}}\right) = \prod_{\vec{\alpha}\in\R^*} f_{\vec{\alpha}}^{\sum_{k=1}^n \frac{m_k}{r_k} \sum_{j=1}^n h_j\alpha_j^*\alpha_{j,k} }.$$
Therefore if $h\not=(0,\dots,0)$, there will be a $\vec{\alpha}\in\R^*$ such that
$$\sum_{k=1}^n \frac{m_k}{r_k} \sum_{j=1}^n h_j\alpha_j^*\alpha_{j,k}\not\in\Z.$$
Hence, every non-trivial element of $H$ gives a non-trivial automorphism of $K(C)$ and $H\subset \Gal(K(C)/K) = \Gal(C)$.

Define
$$F_j^* = \prod_{\vec{\alpha}\in\R^*}f_{\vec{\alpha}}^{\alpha_j^*}, j=1,\dots,n.$$

Since $\vec{\alpha}_j$ has order $s_j$ we get $s_j(\vec{\alpha}_j) = (0,\dots,0)$. Therefore, $s_j\alpha_{j,k} \equiv 0 \Mod{r_k}$ and we can find $\alpha'_{j,k}$ such that
$$\alpha_{j,k} = \frac{r_k}{(s_j,r_k)}\alpha_{j,k}'.$$
Therefore,
$$\sqrt[r_k]{F_k(X)} =  \prod_{\vec{\alpha}\in\R^*}f_{\vec{\alpha}}^{\alpha_k/r_k} = \prod_{\vec{\alpha}\in\R^*}f_{\vec{\alpha}}^{\frac{1}{r_k}\sum_{j=1}^n \alpha_j^*\alpha_{j,k}} = \prod_{j=1}^n \left( \prod_{\vec{\alpha}\in\R^*}f_{\vec{\alpha}}^{\alpha_j^*} \right)^{\alpha_{j,k}/r_k}$$
$$ = \prod_{j=1}^n \left( \prod_{\vec{\alpha}\in\R^*}f_{\vec{\alpha}}^{\alpha_j^*} \right)^{\alpha_{j,k}'/(s_j,r_k)} = \prod_{j=1}^n \left(\sqrt[s_j]{F^*_j(X)}\right)^{\alpha_{j,k}'s_j/(s_j,r_k)}. $$
Hence,
$$K\left(\sqrt[r_1]{F_1(X)},\dots,\sqrt[r_n]{F_n(X)}\right) \subset K\left(\sqrt[s_1]{F^*_1(X)},\dots,\sqrt[s_n]{F^*_n(X)} \right).$$

Clearly $\Gal\left(K\left(\sqrt[s_1]{F^*_1(X)},\dots,\sqrt[s_n]{F^*_n(X)} \right)/K\right) \subset H$. Thus $\Gal(C)\subset H$ and therefore $\Gal(C)=H$.

It remains to show that $g(C)=\frac{g-1}{|G|/|H|}+1$.

Recall, $e(\vec{\alpha}) = \lcm\left(\frac{r_j}{(r_j,\alpha_j)}\right)$. Then $e(\vec{\alpha})$ will be the order of $\vec{\alpha}$ as viewed as an element in $G$. Therefore, if $\vec{\alpha}\in\R^*$, then
$$e(\vec{\alpha}) = \lcm\left(\frac{r_j}{(r_j,\alpha_j)}\right) = \lcm\left(\frac{s_i}{(s_i,\alpha^*_i)}\right):=e^*(\vec{\alpha})$$
since $e^*(\vec{\alpha})$ would the order of $\vec{\alpha}^*$ as viewed as an element in $H$ (which would be the same as $\vec{\alpha}$ in $G$). Therefore,
$$c(\vec{\alpha}) = |G|-\frac{|G|}{e(\vec{\alpha})} = \frac{|G|}{|H|}\left(|H|-\frac{|H|} {e^*(\vec{\alpha})}\right):= \frac{|G|}{|H|} c^*(\vec{\alpha}).$$
Likewise, if we define $d_j^* = \deg(F_j^*) =\sum_{\vec{\alpha}\in\R^*} \alpha^*_jd(\vec{\alpha})$, then $e^*(\vec{d}^*)=e(\vec{d})$ and $\frac{|G|}{|H|}c^*(\vec{d}^*)=c(\vec{d})$.

Since $\vec{d}(\vec{\alpha})$ satisfies \eqref{genform}, we have
$$2g+2|G|-2 =\sum_{\vec{\alpha}\in\R}c(\vec{\alpha})d(\vec{\alpha}) +c(\vec{d}) = \sum_{\vec{\alpha}\in\R^*}c(\vec{\alpha})d(\vec{\alpha}) +c(\vec{d})$$
$$ = \frac{|G|}{|H|}\left(\sum_{\vec{\alpha}\in\R^*}c^*(\vec{\alpha})d(\vec{\alpha}) +c^*(\vec{d}^*)\right).$$
That is,
$$\left(\sum_{\vec{\alpha}\in\R^*}c^*(\vec{\alpha})d(\vec{\alpha}) +c(\vec{d}^*)\right) = 2\left(\frac{g-1}{|G|/|H|}+1\right)+2|H|-2$$
Therefore, $(f_{\vec{\alpha}})$ corresponds to a monic curve $C$ with $\Gal(C) = H$ and, by the Riemann-Hurwitz formula, $g(C)$ is $\frac{g-1}{|G|/|H|}+1$.

\end{proof}

Therefore, for any monic curve with $\Gal(C)=H\subset G$ and $g(C)=\frac{g-1}{|G|/|H|}+1$, we can find $(f_{\vec{\alpha}})\in\F_{\vec{d}(\vec{\alpha})}$ such that $K(C) = K(\sqrt[r_1]{F_1(X)},\dots,\sqrt[r_n]{F_n(X)})$
where
$$F_j(X) = \prod_{\vec{\alpha}\in\R} f_{\vec{\alpha}}(X)^{\alpha_j}.$$

\begin{cor}\label{curvecor1}

For any $\vec{k}\in\R'$ and $E$ as in \eqref{matrix}, let
\begin{align*}
M_{\vec{k},E}(G,g) = \{C, \mbox{ monic } : & \Gal(C)=H\subset G, g(C)=\frac{g-1}{|G|/|H|}+1, \deg{F_j}\equiv k_j \Mod{r_j} \\
& \chi_{r_j}(F_j(x_i)) = \epsilon_{i,j}, i=1\dots, \ell, j=1,\dots,n\}.
\end{align*}
Then there is a natural bijection from elements of $\F_{D;\vec{k},E}$ to $M_{G,g}(\vec{k},E)$ where $D=2g+2|G|-2-c(\vec{k})$.
\end{cor}

\begin{proof}
Follows immediately from Proposition \ref{curveprop1} and the definition of $\F_{D;\vec{k},E}$.
\end{proof}


\section{Inclusion-Exclusion of Abelian Groups}\label{inex}

Therefore, if we can determine the size of $\F_{D;\vec{k},E}$ corresponding to curves with any abelian Galois group and any genus, then we can hope to do an inclusion-exclusion type argument for abelian groups. Luckily, this was first done by Delsarte \cite{dels}.

Let $\mathcal{G}$ be the set of all abelian groups. Define a function
$$\mu:\mathcal{G} \to \Z$$
by
$$\mu\left(\Z/p^{a_1}\Z\times\dots\times\Z/p^{a_n}\Z\right) = \begin{cases} (-1)^n p^{\frac{n(n-1)}{2}} & a_1=\dots=a_n=1 \\ 0 & \mbox{otherwise}  \end{cases}.$$
To finish the definition if $G=G_1\times G_2$ such that $(|G_1|,|G_2|)=1$, then $\mu(G)=\mu(G_1)\mu(G_2)$. Then we have the property that
\begin{align}\label{mobinv}
\sum_{H\subset G} \mu(H) = \begin{cases} 1 & G=\{e\}  \\ 0 & \mbox{otherwise} \end{cases}.
\end{align}

\begin{rem}
This formula requires that we sum up over all subgroups of $G$ in the sense of Remark \ref{curverem}. Hence why it is important that we define $M_{G,g}$ and $M_{G,g}(\vec{k},E)$ in the way that we do.
\end{rem}

For an example of \eqref{mobinv} consider the group $\Z/Q^2\Z$, for $Q$ a prime. Then the subgroups are $\{e\}, \Z/Q\Z$ and $\Z/Q^2\Z$ and each of them appear once. Therefore,
\begin{align*}
\sum_{H\subset \Z/Q^2\Z} \mu(H) & = \mu(\{e\}) + \mu(\Z/Q\Z) + \mu(\Z/Q^2\Z)\\
& = 1 +(-1)+0 =0.
\end{align*}
Whereas if we consider the group $\Z/Q\Z\times \Z/Q\Z$, for $Q$ a prime, then the subgroups would be $\{e\}$, $\Z/Q\Z$ and $\Z/Q\Z\times\Z/Q\Z$. Obviously $\{e\}$ and $\Z/Q\Z\times\Z/Q\Z$ appear only once however, $\Z/Q\Z$ can appear many times. It is easy to see that all the subgroups of $\Z/Q\Z$ lying in $\Z/Q\Z\times \Z/Q\Z$ will be generated by $(1,a)$, $a\in\Z/Q\Z$ or $(0,1)$. That is, there are $Q+1$ different subgroups of $\Z/Q\Z$ appearing in $\Z/Q\Z\times \Z/Q\Z$. Therefore,
\begin{align*}
\sum_{H\subset \Z/Q\Z\times\Z/Q\Z} \mu(H) & = \mu(\{e\}) + (Q+1)\mu(\Z/Q\Z) + \mu(\Z/Q\Z\times\Z/Q\Z)\\
& = 1 +(Q+1)(-1)+Q =0.
\end{align*}

This allows us to perform M$\ddot{o}$bius inversion on $M(G,g)$.

\begin{lem}\label{inexlem}
For any abelian group, $G$, and genus, $g$,
\begin{align*}
|\Hh^*_{G,g}| = \sum_{H\subset G} \mu(G/H)|M(H, \frac{g-1}{|G|/|H|}+1)|
\end{align*}

Likewise, for any $\vec{k}\in\R$ and $E$ as in \eqref{matrix}
\begin{align*}
|\Hh^*_{G,g}(\vec{k},E)| = \sum_{H\subset G} |M_{\vec{k},E}(H,\frac{g-1}{|G|/|H|}+1)|
\end{align*}
\end{lem}

\begin{proof}

Straight from the definition we get
$$|M(G,g)| = \sum_{H\subset G}\left|\Hh^*_{H, \frac{g-1}{|G|/|H|}+1}\right|.$$
Therefore,
\begin{align*}
\sum_{H\subset G} \mu(G/H) \left|M\left(H,\frac{g-1}{|G|/|H|}+1\right)\right| & = \sum_{H\subset G} \mu(G/H) \sum_{H'\subset H} \left|\Hh^*_{H',\frac{g-1}{|G|/|H'|}+1}\right|\\
& = \sum_{H'\subset G } \left|\Hh^*_{H',\frac{g-1}{|G|/|H'|}+1}\right| \sum_{H'\subset H \subset G} \mu(G/H) \\
& = \sum_{H'\subset G } \left|\Hh^*_{H',\frac{g-1}{|G|/|H'|}+1}\right| \sum_{H''\subset G/H'} \mu(H'') \\
& =   |\Hh^*_{G,g}|.
\end{align*}

The proof of the likewise is analogous.
\end{proof}


\section{Generating Series}\label{gensersec}

It remains to determine the size of $\F_{D;\vec{k},E}$ as $D\to\infty$. In order to do this we will develop a generating series for this set. But first, we need indicator functions for the relations
$$d_j \equiv k_j \Mod{r_j}, j=1,\dots,n$$
$$\chi_{r_j}(F(x_i))=\epsilon_{i,j}, i=1,\dots,\ell, j=1,\dots,n.$$

That is, if we let $\xi_{r_j}= e^{\frac{2\pi i}{r_j}}$, a primitive $r_j^{th}$ root of unity, then
\begin{align}\label{indicatfact1}
\frac{1}{r_1\cdots r_n}\prod_{j=1}^n \sum_{t_j=0}^{r_j-1} \xi_{r_j}^{t_j(\sum \alpha_j\deg(f_{\vec{\alpha}}) - k_j)} = \begin{cases} 1 & \sum_{\vec{\alpha}\in\R} \alpha_j\deg(f_{\vec{\alpha}}) \equiv k_j \Mod{r_j} \\ 0 & \mbox{otherwise}\end{cases}.
\end{align}
Further, if we denote $h(X)= \prod_{i=1}^{\ell} (X-x_i)$, then as long as $(F_j,h)=1$ for $j=1,\dots,n$, we get
\begin{align}\label{indicatfact2}
\left(\frac{1}{r_1\cdots r_n}\right)^{\ell}\prod_{i=1}^{\ell}\prod_{j=1}^n \sum_{\nu_{i,j}=0}^{r_j-1} (\epsilon_{i,j}^{-1}\chi_{r_j}(F_j(x_i))^{\nu_{i,j}} = \begin{cases} 1 & \chi_{r_j}(F_j(x_i))=\epsilon_{i,j}, i=1,\dots,\ell, j=1,\dots,n \\ 0 & \mbox{otherwise} \end{cases}.
\end{align}
\begin{rem}
The sum in the exponent in \eqref{indicatfact1} is a sum over all $\vec{\alpha}\in\R$.
\end{rem}

For ease of notation, for every set of polynomials $(f_{\vec{\alpha}})$, let  $I_{\vec{k},E}((f_{\vec{\alpha}}))$ be the indicator function defined as
\begin{align}\label{indicator1}
I_{\vec{k},E}((f_{\vec{\alpha}})) = \left(\frac{1}{r_1\cdots r_n}\right)^{\ell+1}\left(\prod_{j=1}^n\sum_{t_j=0}^{r_j-1} \xi_{r_j}^{t_j(\sum \alpha_j\deg(f_{\vec{\alpha}}) - k_j)}\right) \left(\prod_{i=1}^{\ell}\prod_{j=1}^n \sum_{\nu_{i,j}=0}^{r_j-1} (\epsilon_{i,j}^{-1}\chi_{r_j}(F_j(x_i))^{\nu_{i,j}}\right).
\end{align}

Now, define the multi-variable complex function
\begin{align}\label{genfunc1}
\mathcal{G}_{\vec{k},E}((s_{\vec{\alpha}})) = \sum_{(f_{\vec{\alpha}})} \frac{\mu^2(h\prod_{\vec{\alpha}\in\R} f_{\vec{\alpha}}) I_{\vec{k},E}((f_{\vec{\alpha}}))}{ \prod_{\vec{\alpha}\in\R} |f_{\vec{\alpha}}|^{c(\vec{\alpha})s_{\vec{\alpha}}} }.
\end{align}

\begin{rem}

The sum is over \textit{all} $r_1\cdots r_n -1$-tuples of monic polynomials $(f_{\vec{\alpha}})_{\vec{\alpha}\in\R}$. However, the factor $\mu^2(h\prod_{\vec{\alpha}\in\R} f_{\vec{\alpha}})$ means that it is zero whenever the set of polynomials $(f_{\vec{\alpha}})$ are not square-free and pairwise coprime as well as coprime to $h$ (and thus non-zero at any of the $x_i$). Moreover, as usual, we let $|f_{\vec{\alpha}}|= q^{\deg(f_{\vec{\alpha}})}$.

\end{rem}

Now, if we let $z_{\vec{\alpha}}=q^{-s_{\vec{\alpha}}}$ and define $F_{\vec{k},E}((z_{\vec{\alpha}})) = \mathcal{G}_{\vec{k},E}((q^{-s_{\vec{\alpha}}}))$, then
\begin{align*}
F_{\vec{k},E}((z_{\vec{\alpha}})) & = \sum_{(f_{\vec{\alpha}})} \mu^2(h\prod_{\vec{\alpha}\in\R} f_{\vec{\alpha}}) I_{\vec{k},E}((f_{\vec{\alpha}})) \prod_{\vec{\alpha}\in\R} z_{\vec{\alpha}}^{c(\vec{\alpha})\deg(f_{\vec{\alpha}})} \\
& = \sum_{\substack{ d(\vec{\alpha})=0 \\ \vec{\alpha}\in\R }}^{\infty} |\F_{\vec{d}(\vec{\alpha});\vec{k},E}| \prod_{\vec{\alpha}\in\R} z_{\vec{\alpha}}^{c(\vec{\alpha})d(\vec{\alpha})}.
\end{align*}

With some abuse of notation, if we let $F_{\vec{k},E}(z)$ be the function that sets all the $z_{\vec{\alpha}}=z$ to be the same in $F_{\vec{k},E}((z_{\vec{\alpha}}))$, then we get
\begin{align}
F_{\vec{k},E}(z) &= \sum_{\substack{ d(\vec{\alpha})=0 \\ \vec{\alpha}\in\R }}^{\infty} |\F_{\vec{d}(\vec{\alpha});\vec{k},E}| z^{\sum_{\vec{\alpha}\in\R}c(\vec{\alpha})d(\vec{\alpha})} \label{genfunc2}  \\
& =  \sum_{D=0}^{\infty} |\F_{D;\vec{k},E}|z^D. \nonumber
\end{align}

Ideally, we would like to write $F_{\vec{k},E}(z)$ as an Euler product. However, this is not possible. We can, though, write it as a sum of functions that can be written as a Euler product. But first we need some notation.Let
$$\M := \left\{ \nu = \begin{pmatrix} \nu_{1,1} & \dots & \nu_{1,n} \\ \vdots & & \vdots \\ \nu_{\ell,1} & \dots & \nu_{\ell,n} \end{pmatrix} \in M_{\ell,n} : \nu_{i,j} \in \Z/r_j\Z \right\}.$$
We can define an action on $\R'$ and $E$ by $\M$ by
\begin{align}
\nu\vec{\alpha} & := \begin{pmatrix} \sum_{j=1}^n \frac{r_n}{r_j}\nu_{1,j}\alpha_j \\ \vdots \\ \sum_{j=1}^n \frac{r_n}{r_j} \nu_{\ell,j} \alpha_j \end{pmatrix} \in (\Z/r_n\Z)^{\ell}\\
E^{\nu} & :=  \prod_{i=1}^{\ell} \prod_{j=1}^n \epsilon_{i,j}^{\nu_{i,j}} \in\mu_{r_n}
\end{align}
Moreover, for any $\vec{\alpha},\vec{\beta}\in\R'$ define
\begin{align}
\vec{\alpha}\cdot\vec{\beta} = \sum_{j=1}^n \frac{r_n}{r_j}\alpha_j\beta_j \in \Z/r_n\Z.
\end{align}

With this notation, we can rewrite \eqref{indicatfact1} as
$$\frac{1}{r_1\cdots r_n} \prod_{j=1}^n \sum_{t_j=0}^{r_j-1} \xi_{r_j}^{t_j(\sum \alpha_j\deg(f_{\vec{\alpha}}) - k_j)} = \frac{1}{r_1\cdots r_n} \sum_{\vec{t}\in\R'} \prod_{j=1}^n \xi_{r_j}^{t_j(\sum \alpha_j\deg(f_{\vec{\alpha}}) - k_j)}$$
$$= \frac{1}{r_1\cdots r_n} \sum_{\vec{t}\in\R'} \xi_{r_n}^{-\vec{t}\cdot\vec{k}} \prod_{\vec{\alpha}\in\R} \xi_{r_n}^{\vec{t}\cdot\vec{\alpha}\deg(f_{\vec{\alpha}})}. $$

Recall $h(X) = \prod_{i=1}^{\ell} (X-x_i)$. For every $\nu\in\M$ and $\vec{\alpha}\in\R$, define
\begin{align}\label{character}
\chi^{\nu\vec{\alpha}}_{r_n}(F(X)) = \begin{cases}\prod_{i=1}^{\ell} \chi_{r_n}^{(\nu\vec{\alpha})_i}(F(x_i)) & (F,h)=1 \\ 0 & \mbox{otherwise}\end{cases}.
\end{align}
Then, $\chi^{\nu\vec{\alpha}}_{r_n}$ if a multiplicative character on $\Ff_q[X]$ modulo $h(X)$. Moreover, it will be trivial if and only if $\nu\vec{\alpha}=\vec{0}$. Hence, we can rewrite \eqref{indicatfact2} as
$$\left(\frac{1}{r_1\cdots r_n}\right)^{\ell} \prod_{i=1}^{\ell}\prod_{j=1}^n \sum_{\nu_{i,j}=0}^{r_j-1} (\epsilon_{i,j}^{-1}\chi_{r_j}(F_j(x_i))^{\nu_{i,j}} = \left(\frac{1}{r_1\cdots r_n}\right)^{\ell} \sum_{\nu\in\M} \prod_{i=1}^{\ell}\prod_{j=1}^n (\epsilon_{i,j}^{-1}\chi_{r_j}(F_j(x_i))^{\nu_{i,j}} $$
$$ = \left(\frac{1}{r_1\cdots r_n}\right)^{\ell} \sum_{\nu\in\M} E^{-\nu} \prod_{\vec{\alpha}\in\R} \prod_{i=1}^{\ell}\prod_{j=1}^n \chi^{\nu_{i,j}}_{r_j}(f_{\vec{\alpha}}^{\alpha_j}(x_i)) = \left(\frac{1}{r_1\cdots r_n}\right)^{\ell} \sum_{\nu\in\M} E^{-\nu} \prod_{\vec{\alpha}\in\R} \chi^{\nu\vec{\alpha}}_{r_n}(f_{\vec{\alpha}}(X)).$$

Therefore, we can rewrite the indicator function in \eqref{indicator1} as
\begin{align}\label{indicator2}
I_{\vec{k},E}((f_{\vec{\alpha}})) = \left(\frac{1}{r_1\cdots r_n}\right)^{\ell+1}\sum_{\vec{t}\in\R'} \sum_{\nu\in\M} E^{-\nu} \xi_{r_n}^{-\vec{t}\cdot\vec{k}} \prod_{\vec{\alpha}\in\R} \xi_{r_n}^{\vec{t}\cdot\vec{\alpha}\deg(f_{\vec{\alpha}})} \chi^{\nu\vec{\alpha}}_{r_n}(f_{\vec{\alpha}}(X)).
\end{align}

We can rewrite $F_{\vec{k},E}(z)$ using this new notation.

\begin{align*}
F_{\vec{k},E}(z) & = \sum_{(f_{\vec{\alpha}})} \mu^2(h\prod_{\vec{\alpha}\in\R} f_{\vec{\alpha}}) I_{\vec{k},E}((f_{\vec{\alpha}})) z^{\sum c(\vec{\alpha})\deg(f_{\vec{\alpha}})}\\
& = \left(\frac{1}{r_1\cdots r_n}\right)^{\ell+1} \sum_{(f_{\vec{\alpha}})} \mu^2(h\prod_{\vec{\alpha}\in\R} f_{\vec{\alpha}})\sum_{\vec{t}\in\R'} \sum_{\nu\in\M} E^{-\nu} \xi_{r_n}^{-\vec{t}\cdot\vec{k}} \prod_{\vec{\alpha}\in\R} \left( \chi_{r_n}^{\nu\vec{\alpha}}(f_{\vec{\alpha}}) (\xi_{r_n}^{\vec{t}\cdot\vec{\alpha}} z^{ c(\vec{\alpha})})^{\deg(f_{\vec{\alpha}})}\right) \\
& = \left(\frac{1}{r_1\cdots r_n}\right)^{\ell+1} \sum_{\vec{t}\in\R'} \sum_{\nu\in\M} E^{-\nu} \xi_{r_n}^{-\vec{t}\cdot\vec{k}}  A_{\vec{t},\nu}(z)
\end{align*}
where
\begin{align*}
A_{\vec{t},\nu}(z) & := \sum_{(f_{\vec{\alpha}})} \mu^2(h\prod_{\vec{\alpha}} f_{\vec{\alpha}}) \prod_{\vec{\alpha}\in\R}\left( \chi_{r_n}^{\nu\vec{\alpha}}(f_{\vec{\alpha}}) (\xi_{r_n}^{\vec{t}\cdot\vec{\alpha}} z^{ c(\vec{\alpha})})^{\deg(f_{\vec{\alpha}})}\right).
\end{align*}

\begin{defn}
We call a function $G : \Ff_q[X]^n\to\mathbb{C}$ an \textbf{$n$-dimensional multiplicative} function if
$$G(f_1,\dots,f_n) = \prod_P G(P^{v_P(f_1)},\dots,P^{v_P(f_n)})$$
where the product is over all prime polynomial $P$ dividing $f_1 \cdots f_n$.
\end{defn}

Therefore, if $G$ is an $n$-dimensional multiplicative function, then
$$\sum_{f_1,\dots,f_n} G(f_1,\dots,f_n) = \prod_{P} \left(1 + \sum_{(a_1,\dots,a_n)\not=(0,\dots,0)}G(P^{a_1},\dots,P^{a_n})\right).$$
where the sum is over all monic polynomials in $\Ff_q[X]$ and the product is over all monic prime polynomials.

Now,
$$G((f_{\vec{\alpha}})) =  \mu^2(h\prod_{\vec{\alpha}} f_{\vec{\alpha}}) \prod_{\vec{\alpha}\in\R}\left( \chi_{r_n}^{\nu\vec{\alpha}}(f_{\vec{\alpha}}) (\xi_{r_n}^{\vec{t}\cdot\vec{\alpha}} z^{ c(\vec{\alpha})})^{\deg(f_{\vec{\alpha}})}\right)$$
is an $|R|$-dimensional multiplicative function. Moreover, if $P$ is a prime polynomial coprime to $h$, then
$$G((P^{a_{\vec{\alpha}}})) = \begin{cases}\chi^{\nu\vec{\alpha_0}}_{r_n}(P)(\xi_{r_n}^{\vec{t}\cdot\vec{\alpha_0}} z^{c(\vec{\alpha_0})})^{\deg(P)} & a_{\vec{\alpha_0}}=1 \mbox{ for some } \vec{\alpha_0}, a_{\vec{\beta}}=0 \mbox{ for all } \vec{\beta}\not=\vec{\alpha_0} \\ 0 & \mbox{otherwise}   \end{cases}$$

Therefore,
\begin{align*}
A_{\vec{t},\nu}(z) & = \sum_{\substack{f_{\vec{\alpha}} \\ \vec{\alpha}\in\R}} \mu^2(h\prod_{\vec{\alpha}} f_{\vec{\alpha}}) \prod_{\vec{\alpha}\in\R}\left( \chi_{r_n}^{\nu\vec{\alpha}}(f_{\vec{\alpha}}) (\xi_{r_n}^{\vec{t}\cdot\vec{\alpha}} z^{ c(\vec{\alpha})})^{\deg(f_{\vec{\alpha}})}\right) \\
&= \prod_{\substack{P \\ (P,h)=1}} \left(1 + \sum_{\vec{\alpha}\in\R} \chi^{\nu\vec{\alpha}}_{r_n}(P)(\xi_{r_n}^{\vec{t}\cdot\vec{\alpha}} z^{c(\vec{\alpha})})^{\deg(P)} \right).
\end{align*}

Now, if we let $c_1<c_2<\dots<c_{\eta}$ be the unique values of the $c(\vec{\alpha})$, then we see that if $|z|< q^{-1/c_1}$, then $A_{\vec{t},\nu}(z)$ absolutely converges for all $\vec{t},\nu$ and, hence, so does $F_{\vec{k},E}(z)$. Therefore, we can express $|\F_{D;\vec{k},E}|$ as a contour integral of $F_{\vec{k},E}(z)$.

\begin{prop}\label{firstresprop}

If $c_1 = \min(c(\vec{\alpha}))$ and $0<\delta_1<q^{-1/c_1}$ then let $C_{\delta_1} = \{z\in \mathbb{C} : |z|=\delta_1\}$, oriented counterclockwise. Then
\begin{align}\label{firstres}
\frac{1}{2\pi i}\oint_{C_{\delta_1}} \frac{F_{\vec{k},E}(z)}{z^{D+1}} dz = |\F_{D;\vec{k},E}|.
\end{align}
\end{prop}

\begin{proof}

By \eqref{genfunc2}, we have
$$F_{\vec{k},E}(z) = \sum_{D=0}^{\infty} |\F_{D;\vec{k},E}|z^D.$$
By our discussion above, $\frac{F_{\vec{k},E}(z)}{z^{D+1}}$ has only one pole at $0$ in the region contained in $C_{\delta_1}$ and it's residue is $|\F_{D;\vec{k},E}|$.

\end{proof}


\section{Analytic Continuation of $A_{\vec{t},\nu}(z)$}

In this section, we will calculate an analytic continuation for $A_{\vec{t},\nu}(z)$ for all $\vec{t},\nu$. Then in the next section, we will use this analytic continuation to analyze the poles of $A_{\vec{t},\nu}(z)$.

Recall $h(X) = \prod_{i=1}^{\ell}(X-x_i)$ and define $\R_{\nu} = \{\vec{\alpha}\in\R: \nu\vec{\alpha}=\vec{0}\}$, then the character $\chi_{r_n}^{\nu\vec{\alpha}}$ (as defined in \eqref{character}) will be trivial if and only if $\vec{\alpha}\in\R_{\nu}$. Therefore,
\allowdisplaybreaks
\begin{align*}
A_{\vec{t},\nu}(z) = & \prod_{\substack{P \\ (P,h)=1}} \left(1 + \sum_{\vec{\alpha}\in\R} \chi^{\nu\vec{\alpha}}_{r_n}(P) (\xi_{r_n}^{\vec{t}\cdot\vec{\alpha}} z^{c(\vec{\alpha})})^{\deg(P)} \right) \\
= & \prod_{\substack{P \\ (P,h)=1}} \left( 1 + \sum_{\vec{\alpha}\in\R_{\nu}}(\xi_{r_n}^{\vec{t}\cdot\vec{\alpha}} z^{c(\vec{\alpha})})^{\deg(P)} + \sum_{\vec{\alpha}\not\in\R_{\nu}} \chi_{r_n}^{\nu\vec{\alpha}}(P)(\xi_{r_n}^{\vec{t}\vec{\alpha}}z^{c(\vec{\alpha})})^{\deg(P)}\right)\\
= & \prod_{P} \left( 1 + \sum_{\vec{\alpha}\in\R_{\nu}}(\xi_{r_n}^{\vec{t}\cdot\vec{\alpha}} z^{c(\vec{\alpha})})^{\deg(P)} + \sum_{\vec{\alpha}\not\in\R_{\nu}} \chi_{r_n}^{\nu\vec{\alpha}}(P) (\xi_{r_n}^{\vec{t}\vec{\alpha}}z^{c(\vec{\alpha})})^{\deg(P)}\right) \times\\
& \prod_{P|h}\left(1 + \sum_{\vec{\alpha}\in\R_{\nu}}(\xi_{r_n}^{\vec{t}\cdot\vec{\alpha}} z^{c(\vec{\alpha})})^{\deg(P)}\right)^{-1}\\
= & \prod_{\vec{\alpha}\in\R_{\nu}}\prod_{P}\left(1 + (\xi_{r_n}^{\vec{t}\cdot\vec{\alpha}} z^{c(\vec{\alpha})})^{\deg(P)}\right) H_{\vec{t},\nu}(z)\\
= & \prod_{\vec{\alpha}\in\R_{\nu}} \frac{Z_K(\xi_{r_n}^{\vec{t}\cdot\vec{\alpha}}z^{c(\vec{\alpha})})} {Z_K(\xi_{r_n}^{2\vec{t}\cdot\vec{\alpha}}z^{2c(\vec{\alpha})})} H_{\vec{t},\nu}(z)
\end{align*}
where
$$Z_K(z) = \prod_P \left(1-z^{\deg(P)}\right)^{-1} = (1-qz)^{-1}$$
is the zeta-function of $K$ in the $z$-variable and
$$H_{\vec{t},\nu}(z) = \prod_{P} \left(\frac{ 1 + \sum_{\vec{\alpha}\in\R_{\nu}}(\xi_{r_n}^{\vec{t}\cdot\vec{\alpha}} z^{c(\vec{\alpha})})^{\deg(P)} + \sum_{\vec{\alpha}\not\in\R_{\nu}} \chi_{r_n}^{\nu\vec{\alpha}}(P) (\xi_{r_n}^{\vec{t}\cdot\vec{\alpha}}z^{c(\vec{\alpha})})^{\deg(P)}}{\prod_{\vec{\alpha}\in\R_{\nu}}\left(1 + (\xi_{r_n}^{\vec{t}\cdot\vec{\alpha}} z^{c(\vec{\alpha})})^{\deg(P)}\right)  }\right) \times$$
$$\prod_{P|h}\left(1 + \sum_{\vec{\alpha}\in\R_{\nu}}(\xi_{r_n}^{\vec{t}\cdot\vec{\alpha}} z^{c(\vec{\alpha})})^{\deg(P)}\right)^{-1}.$$

Now, for all $\vec{\alpha}\in\R$,
$$\frac{Z_K(\xi_{r_n}^{\vec{t}\cdot\vec{\alpha}}z^{c(\vec{\alpha})})}{Z_K(\xi_{r_n}^{2\vec{t}\cdot\vec{\alpha}} z^{2c(\vec{\alpha})})}$$
is a meromorphic function with simple poles when $z^{c(\vec{\alpha})}=(q\xi_{r_n}^{\vec{t}\cdot\vec{\alpha}}) ^{-1/c(\vec{\alpha})}$. So it remains to determine where $H_{\vec{t},\nu}(z)$ converges.

\begin{lem}\label{analyticlem}

$H_{\vec{t},\nu}(z)$ absolutely converges for all $|z|<q^{-1/2c_1}$.

\end{lem}

\begin{proof}

Since
$$\prod_{P|h}\left(1 + \sum_{\vec{\alpha}\in\R_{\nu}}(\xi_{r_n}^{\vec{t}\cdot\vec{\alpha}} z^{c(\vec{\alpha})})^{\deg(P)}\right)^{-1}$$
is a finite product, it will always converge and thus we need only consider the factor consisting of the infinite product.

$$\prod_{P} \left(\frac{ 1 + \sum_{\vec{\alpha}\in\R_{\nu}}(\xi_{r_n}^{\vec{t}\cdot\vec{\alpha}} z^{c(\vec{\alpha})})^{\deg(P)} + \sum_{\vec{\alpha}\not\in\R_{\nu}} \chi_{r_n}^{\nu\vec{\alpha}}(P) (\xi_{r_n}^{\vec{t}\cdot\vec{\alpha}}z^{c(\vec{\alpha})})^{\deg(P)}}{\prod_{\vec{\alpha}\in\R_{\nu}}\left(1 + (\xi_{r_n}^{\vec{t}\cdot\vec{\alpha}} z^{c(\vec{\alpha})})^{\deg(P)}\right)  }\right)$$
$$ = \prod_{\vec{\alpha}\not\in\R_\nu} \prod_P \left(1+\chi_{r_n}^{\nu\vec{\alpha}}(P) (\xi_{r_n}^{\vec{t}\cdot\vec{\alpha}}z^{c(\vec{\alpha})})^{\deg(P)}\right) H^*_{\vec{t},\nu}(z). $$

Since, for all $\vec{\alpha}\not\in\R_\nu$, $\chi_{r_n}^{\nu\vec{\alpha}}$ is a non-trivial character we get that
$$ \prod_{\vec{\alpha}\not\in\R_{\nu}}\prod_P \left(1+\chi_{r_n}^{\nu\vec{\alpha}}(P) (\xi_{r_n}^{\vec{t}\cdot\vec{\alpha}}z^{c(\vec{\alpha})})^{\deg(P)}\right)$$
is an entire function. Moreover,
\begin{align*}
H^*_{\vec{t},\nu}(z) & = \prod_{P} \left(\frac{ 1 + \sum_{\vec{\alpha}\in\R_{\nu}}(\xi_{r_n}^{\vec{t}\cdot\vec{\alpha}} z^{c(\vec{\alpha})})^{\deg(P)} + \sum_{\vec{\alpha}\not\in\R_{\nu}} \chi_{r_n}^{\nu\vec{\alpha}}(P) (\xi_{r_n}^{\vec{t}\cdot\vec{\alpha}}z^{c(\vec{\alpha})})^{\deg(P)}}{\prod_{\vec{\alpha}\in\R_\nu}\left(1 + (\xi_{r_n}^{\vec{t}\cdot\vec{\alpha}} z^{c(\vec{\alpha})})^{\deg(P)}\right) \prod_{\vec{\alpha}\not\in\R_\nu} \left(1+\chi_{r_n}^{\nu\vec{\alpha}}(P) (\xi_{r_n}^{\vec{t}\cdot\vec{\alpha}}z^{c(\vec{\alpha})})^{\deg(P)}\right)  }\right)\\
&= \prod_{P} \left(1 - \frac{h_P(z)}{\prod_{\vec{\alpha}\in\R_\nu}\left(1 + (\xi_{r_n}^{\vec{t}\cdot\vec{\alpha}} z^{c(\vec{\alpha})})^{\deg(P)}\right) \prod_{\vec{\alpha}\not\in\R_\nu} \left(1+\chi_{r_n}^{\nu\vec{\alpha}}(P) (\xi_{r_n}^{\vec{t}\cdot\vec{\alpha}}z^{c(\vec{\alpha})})^{\deg(P)}\right)}\right)
\end{align*}
where
\begin{align*}
h_p(z) = &  \prod_{\vec{\alpha}\in\R_\nu}\left(1 + (\xi_{r_n}^{\vec{t}\cdot\vec{\alpha}} z^{c(\vec{\alpha})})^{\deg(P)}\right) \prod_{\vec{\alpha}\not\in\R_\nu} \left(1+\chi_{r_n}^{\nu\vec{\alpha}}(P) (\xi_{r_n}^{\vec{t}\cdot\vec{\alpha}}z^{c(\vec{\alpha})})^{\deg(P)}\right)\\
& - \left(1 + \sum_{\vec{\alpha}\in\R_{\nu}}(\xi_{r_n}^{\vec{t}\cdot\vec{\alpha}} z^{c(\vec{\alpha})})^{\deg(P)} + \sum_{\vec{\alpha}\not\in\R_{\nu}} \chi_{r_n}^{\nu\vec{\alpha}}(P) (\xi_{r_n}^{\vec{t}\cdot\vec{\alpha}}z^{c(\vec{\alpha})})^{\deg(P)}\right)\\
= & O\left(z^{\underset{\vec{\alpha}\not=\vec{\beta}}{\min}(c(\vec{\alpha}) + c(\vec{\beta}))}\right)= O\left(z^{2c_1}\right).
\end{align*}

Therefore, if $|z|<q^{-1/2c_1}$, then $H^*_{\vec{t},\nu}(z)$ converges absolutely and hence so does $H_{\vec{t},\nu}(z)$.

\end{proof}

For $0\leq a \leq r_n-1$, and $i=1,\dots,\eta$, define
\begin{align}\label{Rpoleset}
\R_{\vec{t},\nu;a,i} = \{\vec{\alpha}\in\R_{\nu} : c(\vec{\alpha})=c_i \mbox{ and }\vec{t}\cdot\vec{\alpha}\equiv a \Mod{r_n}\}
\end{align}
and let
\begin{align}\label{poleorder}
m_{\vec{t},\nu;a,i} = |\R_{\vec{t},\nu;a,i}|.
\end{align}

\begin{cor}\label{analyticcor}

$A_{\vec{t},\nu}(z)$ is meromorphic on the disc $|z|<q^{-1/2c_1}$ with poles of order $m_{\vec{t},\nu;a,i}$ at
$$z=\xi_{c_i}^k\left(q\xi_{r_n}^a\right)^{-1/c_i}$$
for $k=1,\dots,c_i$.

\end{cor}

\begin{proof}

Immediate from Lemma \ref{analyticlem} and the factors of $Z_K(z)$ appearing.

\end{proof}

\begin{rem}

It is highly possible that $m_{\vec{t},\nu;a,i}=0$ for some values of $\vec{t},\nu,a,i$. In this case when we say a pole of order $0$, we mean there is no pole.

\end{rem}


\section{Residue Calculations}\label{rescalcsec}

Now, we can calculate the residues of $A_{\vec{t},\nu}(z)$ at each of its poles.

\begin{lem}\label{rescalclem1}

Let $a,i$ be such that $m_{\vec{t},\nu;a,i}\not=0$, then for any $1\leq k \leq c_i$,
$$\Res_{z=\xi_{c_i}^k \left( q\xi_{r_n}^a \right)^{-1/c_i}}\left( \frac{A_{\vec{t},\nu}(z)}{z^{D+1}} \right) = P_{\vec{t},\nu;a,i,k}(D) q^{\frac{D}{c_i}}$$
where $P_{\vec{t},\nu;a,i,k}$ is a quasi-polynomial of degree $(m_{\vec{t},\nu;a,i}-1)$ with leading coefficient $-C_{\vec{t},\nu;a,i,k}$ such that
$$C_{\vec{t},\nu;a,i,k} = \frac{1}{(m_{\vec{t},\nu;a,i}-1)!}\left(\frac{1-q^{-1}}{c_i}\right)^{m_{\vec{t},\nu;a,i}} \xi_{c_i}^{-kD}\left(\xi_{r_n}^a\right)^{\frac{D}{c_i}} H_{\vec{t},\nu;a,i}(\xi^k_{c_i}(\xi_{r_n}^aq)^{-1/c_i})$$
and $H_{\vec{t},\nu;a,i}$ is defined in the proof.
\end{lem}

\begin{proof}

\begin{align*}
\frac{A_{\vec{t},\nu}(z)}{z^{D+1}} & = \frac{1}{z^{D+1}} \prod_{\vec{\alpha}\in\R_{\nu}} \frac{Z_K(\xi_{r_n}^{\vec{t}\cdot\vec{\alpha}}z^{c(\vec{\alpha})})} {Z_K(\xi_{r_n}^{2\vec{t}\cdot\vec{\alpha}}z^{2c(\vec{\alpha})})} H_{\vec{t},\nu}(z) \\
& = \frac{1}{z^{D+1}} \underset{m_{\vec{t},\nu;b,j\not=0}} {\prod_{j=1}^{\eta} \prod_{b=0}^{r_n-1}} \left(\frac{Z_K(\xi_{r_n}^{b}z^{c_j})} {Z_K(\xi_{r_n}^{2b}z^{2c_j})}\right)^ {m_{\vec{t},\nu;b,j}} H_{\vec{t},\nu}(z)\\
& = \frac{1}{z^{D+1}}\left(\frac{1-q\xi_{r_n}^{2a}z^{2c_i}}{1-q\xi_{r_n}^{a}z^{c_i}}\right)^{m_{\vec{t},\nu;a,i}} H_{\vec{t},\nu;a,i}(z)
\end{align*}
where
$$H_{\vec{t},\nu;a,i}(z) = \underset{\substack{(b,j)\not=(a,i) \\ m_{\vec{t},\nu;b,j\not=0}}} {\prod_{j=1}^{\eta} \prod_{b=0}^{r_n-1}}\left(\frac{Z_K(\xi_{r_n}^{b}z^{c_j})} {Z_K(\xi_{r_n}^{2b}z^{2c_j})}\right)^ {m_{\vec{t},\nu;b,j}} H_{\vec{t},\nu}(z).$$

Therefore, for any $1\leq k \leq c_i$, if we let
$$R_{a,i,k}(z) = \frac{z^{c_i}-(q\xi_{r_n}^a)^{-1}}{z-\xi_{c_i}^k(q\xi_{r_n}^a)^{-1/c_i}},$$
then

\begin{align*}
& (m_{\vec{t},\nu;a,i}-1)!\Res_{z=\xi^k_{c_i}(\xi_{r_n}^aq)^{-1/c_i}} \left(\frac{A_{\vec{t},\nu}(z)}{z^{D+1}}\right)\\
= & \lim_{z\to \xi^k_{c_i}(\xi_{r_n}^aq)^{-1/c_i}} \frac{d^{m_{\vec{t},\nu;a,i}-1}}{dz^{m_{\vec{t},\nu;a,i}-1}} \frac{(z-\xi_{c_i}^k(\xi_{r_n}^aq)^{-1/c_i})^{m_{\vec{t},\nu;a,i}}}{z^{D+1}}\left(\frac{1-q\xi_{r_n}^{2a}z^{2c_i}}{1-q\xi_{r_n}^{a} z^{c_i}} \right)^{m_{\vec{t},\nu;a,i}} H_{\vec{t},\nu;a,i}(z) \\
= & \lim_{z\to \xi^k_{c_i}(\xi_{r_n}^aq)^{-1/c_i}} \frac{d^{m_{\vec{t},\nu;a,i}-1}}{dz^{m_{\vec{t},\nu;a,i}-1}} \frac{1}{z^{D+1}}\left(\frac{1-q\xi_{r_n}^{2a}z^{2c_i}}{-q\xi_{r_n}^aR_{a,i,k}(z)} \right)^{m_{\vec{t},\nu;a,i}} H_{\vec{t},\nu;a,i}(z) \\
= & \lim_{z\to \xi^k_{c_i}(\xi_{r_n}^aq)^{-1/c_i}}\sum_{j=0}^{m_{\vec{t},\nu;a,i}-1}\binom{m_{\vec{t},\nu;a,i}-1}{j} \frac{d^j}{dz^j} \left(\frac{1}{z^{D+1}}\right) \frac{d^{m_{\vec{t},\nu;a,i}-1-j}}{dz^{m_{\vec{t},\nu;a,i}-1-j}} \left(\frac{1-q\xi_{r_n}^{2a}z^{2c_i}}{-q\xi_{r_n}^aR_{a,i,k}(z)} \right)^{m_{\vec{t},\nu;a,i}}\times \\
& H_{\vec{t},\nu;a,i}(z) \\
= & \sum_{j=0}^{m_{\vec{t},\nu;a,i}-1}\binom{m_{\vec{t},\nu;a,i}-1}{j} (-1)^j(D+1)\cdots(D+j) \xi_{c_i}^{-k(D+j+1)} (\xi_{r_n}^aq)^{(D+j+1)/c_i} \times \\
& \frac{d^{m_{\vec{t},\nu;a,i}-1-j}}{dz^{m_{\vec{t},\nu;a,i}-1-j}}\left(\frac{1-q\xi_{r_n}^{2a}z^{2c_i}} {-q\xi_{r_n}^aR_{a,i,k} (z)} \right)^{m_{\vec{t},\nu;a,i}} H_{\vec{t},\nu;a,i}(z) \bigg|_{z = \xi^k_{c_i}(\xi_{r_n}^aq)^{-1/c_i}}\\
= & P_{\vec{t},\nu;a,i,k}(D)q^{\frac{D}{c_i}}
\end{align*}
where $P_{\vec{t},\nu;a,i,k}$ is a quasi-polynomial of degree $m_{\vec{t},\nu;a,i}-1$. Moreover, we see that the leading coefficient of $P_{\vec{t},\nu;a,i,k}$ arises when $j=m_{\vec{t},\nu;a,i}-1$. That is
\begin{align*}
P_{\vec{t},\nu;a,i,k}(D)q^{\frac{D}{c_i}} = & (-D)^{m_{\vec{t},\nu;a,i}-1} \xi_{c_i}^{-k(D+m_{\vec{t},\nu;a,i})} \left(\xi^a_{r_n}q\right)^{(D+m_{\vec{t},\nu;a,i})/c_i} \left(\frac{1-q^{-1}}{-c_i\xi_{c_i}^{k(c_i-1)} \left(\xi_{r_n}^aq \right)^{1/c_i} }\right)^{m_{\vec{t},\nu;a,i}} \times \\
& H_{\vec{t},\nu;a,i}(\xi^k_{c_i}(\xi_{r_n}^aq)^{-1/c_i})\left(1+O\left(\frac{1}{D}\right)\right) \\
= & -\left(\frac{1-q^{-1}}{c_i}\right)^{m_{\vec{t},\nu;a,i}} \xi_{c_i}^{-kD}  D^{m_{\vec{t},\nu;a,i}-1} \left(\xi_{r_n}^aq\right)^{\frac{D}{c_i}} H_{\vec{t},\nu;a,i}(\xi^k_{c_i}(\xi_{r_n}^aq)^{-1/c_i}) \times \\
& \left(1+O\left(\frac{1}{D}\right)\right).
\end{align*}

\end{proof}

\begin{cor}\label{rescalccor1}

Let $m_{\vec{t},\nu,i} = \underset{0\leq a \leq r_n-1}{\max}(m_{\vec{t},\nu;a,i})$. Let $0<\delta_1<q^{-1/c_1}$, $\delta_2 = \frac{1+\epsilon}{2c_1}$ for some $\epsilon>0$ and let $C_{\delta_1} = \{z\in\mathbb{C} : |z|=\delta_1\}$ oriented counterclockwise and $C_{\delta_2} = \{z\in\mathbb{C}: |z| = q^{-\delta_2}\}$ oriented clockwise. Then
$$\frac{1}{2\pi i} \oint_{C_{\delta_1}+C_{\delta_2}} \frac{A_{\vec{t},\nu}(z)}{z^{D+1}} dz = \sum_{i=1}^{\eta}P_{\vec{t},\nu,i}(D)q^{\frac{D}{c_i}} $$
where $P_{\vec{t},\nu,i}$ is a quasi-polynomial such that
$$P_{\vec{t},\nu,i}(D) = C_{\vec{t},\nu,i}D^{m_{\vec{t},\nu;i}-1} + O\left(D^{m_{\vec{t},\nu;i}-2}\right) $$
with
$$C_{\vec{t},\nu,i} = \sum_{\substack{a=0 \\ m_{\vec{t},\nu;a,i} = m_{\vec{t},\nu;i}}}^{r_n-1} \sum_{k=0}^{c_i-1} C_{\vec{t},\nu;a,i,k}.$$

\end{cor}

\begin{proof}\label{rescalccor2}

By Cauchy's Residue Theorem, and the fact that the larger disc, $C_{\delta_2}$, is oriented clockwise,
\begin{align*}
\frac{1}{2\pi i} \oint_{C_{\delta_1}+C_{\delta_2}} \frac{A_{\vec{t},\nu}(z)}{z^{D+1}} dz & = \underset{m_{\vec{t},\nu;a,i}\not=0} {\sum_{i=1}^{\eta} \sum_{a=0}^{r_n-1}} \sum_{k=0}^{c_i} -\Res_{z=\xi_{c_i}^k \left( q\xi_{r_n}^a \right)^{-1/c_i}}\left( \frac{A_{\vec{t},\nu}(z)}{z^{D+1}} \right)\\
& = \underset{m_{\vec{t},\nu;a,i}\not=0} {\sum_{i=1}^{\eta} \sum_{a=0}^{r_n-1} } \sum_{k=0}^{c_i} -P_{\vec{t},\nu;a,i,k}(D) q^{\frac{D}{c_i}}\\
& = \sum_{i=1}^{\eta}P_{\vec{t},\nu,i}(D)q^{\frac{D}{c_i}}.
\end{align*}
The fact that $P_{\vec{t},\nu,i}(D)$ satisfies the conditions in the statement follow directly from Lemma \ref{rescalclem1}
\end{proof}

\begin{rem}

Now, we are unable to determine if $C_{\vec{t},\nu,i}$ is non-zero. Hence we can only give a bound on the degree of the $P_{\vec{t},\nu,i}$. This is why in the statement of the main theorems we say "of degree at most" instead of give the exact degree.

\end{rem}

\begin{prop}\label{rescalcprop1}
Let
$$m_i = \underset{\substack{ \vec{t}\in\R \\ \nu\in\M }}{\max}(m_{\vec{t},\nu,i}).$$
If there exists a solution to \eqref{genform2} and \eqref{equiv1}, then for every $\epsilon>0$,
$$|\F_{D;\vec{k},E}| = \sum_{i=1}^{\eta} P_i(D)q^{\frac{D}{c_i}} + O\left(q^{(\frac{1}{2}+\epsilon)\frac{D}{c_1}}\right)$$
where $P_i$ is a quasi-polynomial such that of degree at most $(m_i-1)$. Otherwise, if there does not exist a solution to \eqref{genform2} and \eqref{equiv1}, then $|\F_{D,\vec{k},E}|=0$.
\end{prop}

\begin{proof}

Recall that
$$\F_{D;\vec{k},E} = \bigcup_{\vec{d}(\vec{\alpha})} \F_{\vec{d}(\vec{\alpha});\vec{k},E}$$
where the union is over all solutions to \eqref{genform2} where $D=2g+2|G|-2-c(\vec{k})$. Therefore, if there are no solutions to \eqref{genform2}, we have an empty union, so $\F_{D;\vec{k},E}=\emptyset$. Further, if there are solution to \eqref{genform2} but none of which that satisfy \eqref{equiv1} then $\F_{D;\vec{k},E}$ would be a union of empty sets and thus empty itself. Therefore, from now on, we will always assume there is a solution to \eqref{genform2} and \eqref{equiv1}.

Let $C_{\delta_1}$ and $C_{\delta_2}$ be as defined in Corollary \ref{rescalccor1}. Then
\begin{align*}
\frac{1}{2\pi i}\oint_{C_{\delta_1}+C_{\delta_2}} \frac{F_{\vec{k},E}(z)}{z^{D+1}} dz & = \left(\frac{1}{r_1\cdots r_n}\right)^{\ell+1} \sum_{\vec{t}\in\R'} \sum_{\nu\in\M} \xi_{r_n}^{-\vec{t}\cdot\vec{k}}E^{-\nu} \frac{1}{2\pi i}\oint_{C_{\delta_1}+C_{\delta_2}} \frac{A_{\vec{t},\nu}(z)}{z^{D+1}} dz \\
& =  \left(\frac{1}{r_1\cdots r_n}\right)^{\ell+1} \sum_{\vec{t}\in\R'} \sum_{\nu\in\M} \xi_{r_n}^{-\vec{t}\cdot\vec{k}}E^{-\nu} \sum_{i=1}^{\eta}P_{\vec{t},\nu,i}(D)q^{\frac{D}{c_i}}\\
& = \sum_{i=1}^{\eta} P_i(D)q^{\frac{D}{c_i}}
\end{align*}
where $P_i$ is a quasi-polynomial of degree at most $m_i$.

Now, by Proposition \ref{firstresprop}, we know that
$$\frac{1}{2\pi i}\oint_{C_{\delta_1}}\frac{F_{\vec{k},E}(z)}{z^{D+1}}dz = -|\F_{D;\vec{k},E}|.$$
Moreover,
$$\left|\frac{1}{2\pi i}\oint_{C_{\delta_2}} \frac{F_{\vec{k},E}(z)}{z^{D+1}}dz \right| = O\left( q^{(\frac{1}{2}+\epsilon) \frac{D}{c_1}} \right) $$
where the implied constant is the maximum values of $F_{\vec{k},E}(z)$ on $C_{\delta_2}$.
\end{proof}

\begin{rem}

If we let $c(\vec{\alpha})$ be any integers, then we could have that $\frac{D}{c_i} \leq \frac{D}{2c_1}$ and thus part of the main term could be absorbed into the error term. However, if we let $c(\vec{\alpha}) = |G|-\frac{|G|}{e(\vec{\alpha})}$, then we actually have that $\frac{D}{c_i}>\frac{D}{2c_1}$ for all $i=1,\dots,\eta$. So for small enough $\epsilon$, none of our main terms can be absorbed into the error term.

\end{rem}


\section{Proofs of the Main Theorems} \label{proofsec}

All that remains is to combine Proposition \ref{rescalcprop1} and Lemma \ref{inexlem}. From now on, we will fix the $c(\vec{\alpha})=|G|-\frac{|G|}{e(\vec{\alpha})}$. But first, we present a little more notation in order to deal with the subgroups of $G$ as used in Lemma \ref{inexlem}.

As in Section \ref{curvesec}, there is a natural bijection from $G\setminus\{e\}$ to $\R$. For every $H\subset G$, let $\R_H$ be the image of $H$ under this natural bijection. Recall that $\eta=\eta_G$ is the number of non-trivial divisors of $\exp(G)=r_n$. Then, for all $\vec{t}\in\R'$, $\nu\in\M$, $0\leq a \leq r_n-1$ and $1\leq i \leq \eta_G$, define the analogous objects
\begin{align*}
\R_{H,\nu} & = \{\vec{\alpha}\in\R_H: \nu\vec{\alpha}=0\}\\
\R_{H,\vec{t},\nu;a,i} & = \{\vec{\alpha}\in\R_{H,\nu} : c(\vec{\alpha})=c_i \mbox{ and } \vec{t}\cdot\vec{\alpha}\equiv a \Mod{r_n}\} \\
m_{H,t,\nu;a,i} & = |\R_{H,\vec{t},\nu;a,i}|\\
m_{H,\vec{t},\nu,i} & = \underset{0\leq a \leq r_n-1}{\max}(m_{H,\vec{t},\nu;a,i})\\
m_{H,i} & = \underset{\substack{\vec{t}\in\R' \\ \nu\in\M}}{\max}(m_{H,\vec{t},\nu,i})
\end{align*}

Now,
$$m_{H,i} = m_{H,0,0;0,i} =  |\{\vec{\alpha}\in\R_H: c(\vec{\alpha})=c_i\}| = \phi_{H}(s_i)$$
since $c_i = |G|-\frac{|G|}{e(\vec{\alpha})}$ and $e(\vec{\alpha})$ is the order of $\vec{\alpha}$ as seen as an element in $G$. So, if $\vec{\alpha}\in\R_H$, then it can be seen as element in $H$ and will have the same order. Notice, however, that we could have $\phi_H(s)=0$ even if $\phi_G(s)\not=0$.

\begin{proof}[Proof of Theorem \ref{mainthm2}]

Lemma \ref{inex} and Corollary \ref{curvecor1} tell us that
$$\Hh^*_{G,g}(\vec{k},E) = \sum_{H \subset G} \sum_{\vec{d}(\vec{\alpha})} \F_{\vec{d}(\vec{\alpha});\vec{k},E}$$
where the inner sum is over all $\vec{d}(\vec{\alpha})$ that satisfy
\begin{align}\label{genform10}
&d(\vec{\alpha}) =0, \vec{\alpha}\not\in\R_H\nonumber \\
&d_j = \sum_{\vec{\alpha}\in\R}\alpha_jd(\vec{\alpha}) \equiv k_j \Mod{r_j}, j=1,\dots,n  \\
&\sum_{\vec{\alpha}\in\R} c(\vec{\alpha})d(\vec{\alpha}) = 2g+2|G|-2-c(\vec{k}). \nonumber
\end{align}

Therefore, if there are no solutions to \eqref{genform1} and \eqref{equiv1}, then the above sum is empty and we have that $|\Hh_{G,g}(\vec{k},E)|=|\Hh^*_{G,g}(\vec{k},E)| = 0$. From now on, we will assume that there exists a solution to \eqref{genform1} and \eqref{equiv1} so that the above sum is non-empty. Further, note that if $g\not\equiv 1 \Mod{|G|/|H|}$ for some $H$ then there would be no solutions to \eqref{genform10} as this would correspond to a curve with a non-integer genus, which is impossible.

Moreover, if $H\cong \Z/s_1\Z\times\dots\Z/s_n\Z$ where $s_j|r_j$, then $\R_H$ can be identified with the set
$$[0,\dots,s_1-1]\times\dots\times[0,\dots,s_n-1]\setminus\{(0,\dots,0)\}.$$
This allows us to apply Proposition \ref{rescalcprop1} to obtain
\begin{align*}
|\Hh^*_{G,g}(\vec{k},E)|&  =\sum_{H\subset G} \mu(G/H)\left( \sum_{j=1}^{\eta_H} P_{H,j;\vec{k},E}(2g)q^{\frac{2g+2|G|-2}{|G|-\frac{|G|}{s_{H,j}}}} + O\left(q^{\frac{(1+\epsilon)g}{|G|-\frac{|G|}{s_{H,1}}}}\right) \right)
\end{align*}
where $\eta_H$ is the number of non-trivial divisors of $\exp(H)$ and $1=s_{H,0}<s_{H,1}<\dots<s_{H,\eta_H}=\exp(H)$ are the divisor of $\exp(H)$ and $P_{H,j;\vec{k},E}$ is a quasi-polynomial of degree at most $\phi_H(s_{H,j})-1$ if $g\equiv 1 \Mod{|G|/|H|}$ and identically the $0$ polynomial otherwise.  Since $\exp(H)|\exp(G)$ for all $H\subset G$ and $\phi_H(s_{H,j})\leq \phi_G(s_{G,j})$, we can write
\begin{align*}
|\Hh_{G,g}(\vec{k},E)|& =\sum_{j=1}^{\eta} P_{j;\vec{k},E}(2g)q^{\frac{2g+2|G|-2}{c_j}} + O\left(q^{\frac{(1+\epsilon)g}{c_1}} \right)
\end{align*}
where $c_j$ and $\eta=\eta_G$ are as above and $P_{j;\vec{k},E}$ is a quasi-polynomial of degree at most $\phi_G(s_j)-1$.

\end{proof}

\begin{proof}[Proof of Theorem \ref{mainthm1}]

If we set $\ell=0$, then we get $E=\emptyset$ is an empty matrix and thus the condition on it vanishes in $\Hh_{\vec{k},\emptyset}(G,g)$. Therefore,
\begin{align*}
|\Hh(G,g)| & = \sum_{\vec{k}\in\R'} |\Hh_{\vec{k},\emptyset}(G,g)|\\
& = \sum_{\vec{k}\in\R'} \sum_{j=1}^{\eta} P_{j;\vec{k},\emptyset}(2g) q^{\frac{2g+2|G|-2}{c_j}} + O\left(q^{\frac{(1+\epsilon)g}{c_1}} \right)\\
& = \sum_{j=1}^{\eta}\sum_{\vec{k}\in\R'} P_{j;\vec{k},\emptyset}(2g )q^{\frac{2g+2|G|-2}{c_j}} + O\left(q^{\frac{(1+\epsilon)g}{c_1}} \right) \\
& = \sum_{j=1}^{\eta}P_{j}(2g)q^{\frac{2g+2|G|-2}{c_j}} + O\left(q^{\frac{(1+\epsilon)g}{c_1}} \right)
\end{align*}
where $P_j$ is a quasi-polynomial of degree at most $\phi_G(s_j)-1$. To show that $P_1$ has exact degree, suppose
$$P_1(2g) = a_0(g)(2g)^{\phi_G(s_j)-1} + a_1(g)(2g)^{\phi_G(s_j-2)} + \dots$$
for some periodic function $a_i$ with integer period. Now, our results, shows that
$$\sum_{j=0}^{c_1-1}q^{-\frac{j}{c_1}}|\Hh_{G,g+j}| \sim \sum_{j=0}^{N-1} a_0(g+j) (2g)^{\phi_G(s_1)-1}q^{\frac{2g+2|G|-2}{c_1}}.$$
Further, \eqref{wright} then tells us that
$$\sum_{j=0}^{c_1-1} a_0(g+j) = C(K,G)\not=0.$$
Therefore, $a_0$ will be non-zero for at least one integer in every interval of length $c_1$. That is, $a_0$ is not identically $0$ and $P_1$ has degree exactly $\phi_G(s_1)-1$.

\end{proof}

\section{$G=(\Z/Q\Z)^n$}\label{Q^n}

In this section we will determine the leading coefficient of $P_{\vec{k},E,1}$ and $P_1$ that appear in Corollaries \ref{mainthm1cor} and \ref{mainthm2cor} in the case that $G=(\Z/Q\Z)^n$.

The reason we are able to determine the leading coefficient of $P_1$ in this case is that the genus and M$\ddot{o}$bius inversion formulas become simpler when $G=(\Z/Q\Z)^n$. Indeed, in this case \eqref{genform} becomes
\begin{align}\label{genformQ}
2g+2Q^n-2 = \begin{cases} (Q^n-Q^{n-1})\sum_{\vec{\alpha}\in\R}d(\vec{\alpha}) & d_j\equiv 0 \Mod{Q}, j=1,\dots,n \\   (Q^n-Q^{n-1})(\sum_{\vec{\alpha}\in\R}d(\vec{\alpha}) +1) & \mbox{otherwise} \end{cases}
\end{align}

Therefore, by Theorem \ref{mainthm2}, we get that if $2g+2Q^n-2\equiv 0 \Mod{Q^n-Q^{n-1}}$ then
$$\Hh_{(\Z/Q\Z)^n,g}(\vec{k},E) = \begin{cases} P_{\vec{0},E}(2g)q^{\frac{2g+2Q^n-2}{Q^n-Q^{n-1}}} & \vec{k}=\vec{0} \\ P_{\vec{k},E}(2g)q^{\frac{2g+2Q^n-2}{Q^n-Q^{n-1}}-1} & \vec{k}\not=\vec{0} \end{cases} + O\left(q^{(1+\epsilon)\frac{g}{Q^n-Q^{n-1}}}\right)$$
for some quasi-polynomial $P_{\vec{k},E}$ whose degree is at most $\phi_G(Q)-1=Q^n-2$. In fact we will show the is, in fact, a polynomials and it has exact degree $Q^n-2$. For the rest of this section we will always be assuming that $2g+2Q^n-2 \equiv 0 \Mod{Q^n-Q^{n-1}}$.

We see that in this case we get $c(\vec{\alpha})=c(\vec{d}) = Q^n-Q^{n-1}$ for all $\vec{\alpha}\in\R$. Therefore, since we always assumed $c(\vec{\alpha})$ was arbitrary we can apply the results therein to this case $c(\vec{\alpha})=c(\vec{d})=1$ and $D=\frac{2g+2Q^n-2}{Q^n-Q^{n-1}}\in\mathbb{N}$ in order to find the leading coefficient of $P_{\vec{k},E}$.

If we look back at where the quasi-polynomials come from, it is because we have a factor of the form $\zeta_{c_i}^{-kD}$ appearing in the constant $C_{\vec{t},\nu;a,i,k}$ in Lemma \ref{rescalclem1}. Therefore, since in the case we can assume $c_i=1$ for all $i$, these terms disappear and we in fact get that $P_{\vec{k},E}$ is a polynomial and not a quasi-polynomial.

\begin{rem}
By setting $D=\frac{2g+2Q^n-2}{Q^n-Q^{n-1}}$ instead of just $2g+2Q^n-2$, we are now counting by \textit{conductor} instead of discriminant (genus). This is more analogous to what Bucur, et al. did in \cite{BDFK+}. Because we can easily switch to counting by conductor is why it is easier to compute the constant in this case
\end{rem}

In this setting, for all $\vec{t}\in\R'$ and $\nu\in\M$, we have that $A_{\vec{t},\nu}(z)$ will have poles of order $m_{\vec{t},\nu;a}$ when $z=(q\xi_Q^a)^{-1}$ where
$$m_{\vec{t},\nu;a} = |\{\vec{\alpha}\in\R_{\nu} : \vec{t}\cdot\vec{\alpha}\equiv a \Mod{Q}\}|.$$
Now, since $\left(\Z/Q\Z\right)^n$ can be viewed as a vector space over the field $\Z/Q\Z$, we get that the action of $\nu$ and $\vec{t}$ on $\R$ are vector space morphisms. Therefore, the set
$$\{\vec{\alpha}\in\R_{\nu} : \vec{t}\cdot\vec{\alpha}\equiv a \Mod{Q}\}\subsetneq \R$$
unless $\nu=0$, $\vec{t}=\vec{0}$ and $a=0$. In which case we get
$$m_{\vec{0},0;0} = |\R|=Q^n-1.$$

Therefore, combining the results of Section \ref{rescalcsec}, we get that the leading coefficient of $P_{\vec{k},E}$ is
$$C_{\vec{k},E} = \frac{1}{(Q^n-2)!}\left(1-q^{-1}\right)^{Q^n-1}\prod_{P}\left(\frac{ |P|^{Q^n-1}+(Q^n-1) |P|^{Q^n-2}}{(|P|+1)^{Q^n-1}} \right) \left(\frac{q}{Q^n(q+Q^n-1)}\right)^{\ell} $$
$$ = \frac{1}{(Q^n-2)!} \frac{L_{Q^n-2}}{\zeta_q(2)^{Q^n-1}} \left(\frac{q}{Q^n(q+Q^n-1)}\right)^{\ell}$$
where
\begin{align}\label{Lnum}
L_m = \prod_{j=1}^m \prod_{P}\left(1 - \frac{j}{(|P|-1)(|P|+j)}\right)
\end{align}
and
\begin{align}\label{zeta}
\zeta_q(s) = \sum_{F\in\Ff_q[X]} \frac{1}{|F|^s} = \frac{1}{1-q^{1-s}}
\end{align}
is the zeta function where $|F|=q^{\deg{F}}$.

Notice that $C_{\vec{k},E}$ does not depend on $\vec{k}$ or $E$. Therefore, if we set $\ell=0$ and sum over all $\vec{k}$, we get
\begin{align*}
\Hh_{(\Z/Q\Z)^n,g}  = & \sum_{\vec{k}\in\R'} \Hh_{(\Z/Q\Z)^n,g)}(\vec{k},\emptyset) \\
= & P_{\vec{0},\emptyset}\left(\frac{2g+2Q^n-2}{Q^n-Q^{n-1}}\right)q^{\frac{2g+2Q^n-2}{Q^n-Q^{n-1}}} + \sum_{\vec{k}\in\R} P_{\vec{k},\emptyset}\left(\frac{2g+2Q^n-2}{Q^n-Q^{n-1}}\right)q^{\frac{2g+2Q^n-2}{Q^n-Q^{n-1}}-1}\\
& + O\left(q^{(1+\epsilon)\frac{g}{Q^n-Q^{n-1}}}\right)\\
= & P\left(\frac{2g+2Q^n-2}{Q^n-Q^{n-1}}\right)q^{\frac{2g+2Q^n-2}{Q^n-Q^{n-1}}} + O\left(q^{(1+\epsilon)\frac{g}{Q^n-Q^{n-1}}}\right)
\end{align*}
where $P$ is a polynomial of degree $Q^n-2$ with leading coefficient
\begin{align*}
C & = C_{0,\emptyset} + \sum_{\vec{k}\in\R}C_{\vec{k},\emptyset}q^{-1}\\
& = \frac{1}{(Q^n-2)!}\frac{q+Q^n-1}{q} \frac{L_{Q^n-2}}{\zeta_q(2)^{Q^n-1}}
\end{align*}
which is exactly the analogue of the constant in \cite{BDFK+}.

Since the condition $F_j(x_{q+1})\not=0$, where $x_{q+1}$ is the point at infinity, is equivalent to saying $\deg(F_j)\equiv 0 \Mod{r_j}$ for $j=1,\dots,n$, we get that if $\epsilon_{i,j}\in\mu_{r_j}$ for $i=1,\dots,q+1$ and $j=1,\dots,n$. Then as $g\to\infty$

\begin{align*}
& \frac{|\{C\in\Hh_{G,g} : \chi_{r_j}(F_j(x_i))=\epsilon_{i,j}, i=1,\dots,q+1, j=1,\dots,n\}|}{|\Hh_{G,g}|}\\
= & \frac{\frac{1}{Q^n}|\Hh_{G,g}(0,E)|}{|\Hh_{G,g}|} = \left(\frac{q}{Q^n(q+Q^n-1)}\right)^{q+1} \left(1+O\left(\frac{1}{g}\right)\right)
\end{align*}
where the $\frac{1}{Q^n}$ factor in the first equality comes from the fact the leading coefficients of the $F_j$ must satisfy $\chi_{r_j}(c_j)=\epsilon_{q+1,j}$.

Finally, from this result, the exact same argument will work to show that as $g\to\infty$,
$$\frac{|\{C\in\Hh_{G,g} : \#C(\Pp^1(\Ff_q)) = M\}|}{|\Hh_{G,g}|} = \Prob\left(\sum_{i=1}^{q+1} X_i= M\right)\left(1+O\left(\frac{1}{g} \right)\right)$$
where the $X_i$ are $i.i.d.$ random variables taking value $0$, $Q^n$ or $Q^{n-1}$  such that
$$X_i = \begin{cases} Q^{n-1} & \mbox{with probability }  \frac{Q^n-1} {Q^{n-1}(q+Q^n -1)} \\ Q^n & \mbox{with probability }  \frac{q}{Q^n(q+Q^n-1)} \\ 0 & \mbox{with probability } \frac{(Q^n-1)(q+Q^n-Q)}{Q^n(q+Q^n-1)}    \end{cases}.$$

\bibliography{FullSpace}

\providecommand{\bysame}{\leavevmode\hbox to3em{\hrulefill}\thinspace}
\providecommand{\MR}{\relax\ifhmode\unskip\space\fi MR }
\providecommand{\MRhref}[2]{%
  \href{http://www.ams.org/mathscinet-getitem?mr=#1}{#2}
}
\providecommand{\href}[2]{#2}
\begin{thebibliography}{10}

\bibitem{BDFK+}
Alina Bucur, Chantal David, Brooke Feigon, Nathan Kaplan, Matilde Lal{\i}n,
  Ekin Ozman, and Melanie~Mathett Wood, \emph{The distribution of points on
  cyclic covers of genus g}, preprint (2015).

\bibitem{BDFL1}
Alina Bucur, Chantal David, Brooke Feigon, and Matilde Lal{\i}n, \emph{Biased
  statistics for traces of cyclic p-fold covers over finite fields}, WIN--Women
  in Numbers: Research Directions in Number Theory \textbf{60} (2009),
  121--143.

\bibitem{BDFL2}
Alina Bucur, Chantal David, Brooke Feigon, and Matilde Lal{\'\i}n,
  \emph{Statistics for traces of cyclic trigonal curves over finite fields},
  International Mathematics Research Notices (2009), rnp162.

\bibitem{dels}
S~Delsarte, \emph{Fonctions de mobius sur les groupes abeliens finis}, Annals
  of Mathematics (1948), 600--609.

\bibitem{DF}
David~Steven Dummit and Richard~M Foote, \emph{Abstract algebra}, vol. 1984,
  Wiley Hoboken, 2004.

\bibitem{hart}
Robin Hartshorne, \emph{Algebraic geometry}, vol.~52, Springer Science \&
  Business Media, 1977.

\bibitem{KS}
Nicholas~M Katz and Peter Sarnak, \emph{Random matrices, frobenius eigenvalues,
  and monodromy}, vol.~45, American Mathematical Soc., 1999.

\bibitem{KR}
P{\"a}r Kurlberg and Ze{\'e}v Rudnick, \emph{The fluctuations in the number of
  points on a hyperelliptic curve over a finite field}, Journal of Number
  Theory \textbf{129} (2009), no.~3, 580--587.

\bibitem{LMM}
Elisa Lorenzo, Giulio Meleleo, Piermarco Milione, and Alina Bucur,
  \emph{Statistics for biquadratic covers of the projective line over finite
  fields}, arXiv preprint arXiv:1503.03276 (2015).

\bibitem{M1}
Patrick Meisner, \emph{Distribution of points on cyclic curves over finite
  fields}, arXiv preprint arXiv:1511.07814 (submitted) (2015).

\bibitem{M2}
\bysame, \emph{Distribution of points on abelian curves over finite fields},
  arXiv preprint arXiv:1612.03411 (2016).

\bibitem{rose}
Michael Rosen, \emph{Number theory in function fields}, vol. 210, Springer
  Science \& Business Media, 2013.

\bibitem{wright}
David~J Wright, \emph{Distribution of discriminants of abelian extensions},
  Proceedings of the London Mathematical Society \textbf{3} (1989), no.~1,
  17--50.

\end{thebibliography}
\bibliographystyle{amsplain}

\end{document}